\documentclass{article}
\usepackage[utf8]{inputenc}
\usepackage{bbold}
\usepackage{anysize}
\marginsize{1in}{1in}{1in}{1in}
\title{The Euler Totient Function on Lucas Sequences}
\author{J.C. Saunders}
\usepackage{amssymb}
\usepackage{amsmath}
\usepackage{amsthm}
\usepackage{mathtools}
\newtheorem{definition}{Definition}
\newtheorem{theorem}{Theorem}
\newtheorem{lemma}{Lemma}

\newtheorem{notation}{Notation}

\newtheorem{case}{Case}

\newtheorem{remark}{Remark}
\theoremstyle{definition}
\begin{document}
\maketitle
\begin{abstract}
In 2009, Luca and Nicolae \cite{luca2} proved that the only Fibonacci numbers whose Euler totient function is another Fibonacci number are $1,2$, and $3$. In 2015, Faye and Luca \cite{faye2} proved that the only Pell numbers whose Euler totient function is another Pell number are $1$ and $2$. Here we add to these two results and prove that for any fixed natural number $P\geq 3$, if we define the sequence $\left(u_n\right)_n$ as $u_0=0$, $u_1=1$, and $u_n=Pu_{n-1}+u_{n-2}$ for all $n\geq 2$, then the only solution to the Diophantine equation $\varphi\left(u_n\right)=u_m$ is $\varphi\left(u_1\right)=\varphi(1)=1=u_1$.
\end{abstract}
\section{Introduction}
There has been much study on Diophantine equations involving the Euler totient function $\varphi$, which counts the number of positive integers up to a certain positive integer $n$ that are coprime up to $n$, and \textit{binary recurrence sequences}. A \textit{binary recurrence sequence} is a sequence of integers that satisfies a given recursion relation of the form $u_n=Pu_{n-1}+Qu_{n-2}$ where $P$ and $Q$ are fixed integers and $\left(u_n\right)_n$ is the sequence in question. There are two important kinds of binary recurrence sequences. A \textit{Lucas sequence of the first kind} is a binary recurrence sequence $\left(u_n\right)_n$, starting with $u_0=0$ and $u_1=1$. A \textit{Lucas sequence of the second kind} is a binary recurrence sequence $\left(v_n\right)_n$, starting with $v_0=2$ and $v_1=P$. There are many results on Diophantine equations of the form $\varphi\left(u_n\right)=v_m$ where $\left(u_n\right)_n$ and $\left(v_n\right)_n$ are Lucas sequences of the first or second kind. For instance, in 2002, Luca gave a general result on such an equation, proving that if a few technical properties of the sequences in question hold, then the number of solutions to such an equation is finite with the upper bound being computable \cite{luca4}. Since then, there have been a number of other similar results, in particular, when terms of a Lucas sequence get mapped to terms in the same Lucas sequence, see \cite{bai}, \cite{chen}, \cite{faye4}, \cite{faye3}, \cite{faye2}, \cite{luca}, and \cite{luca2}. There are also such results and study on the occurrence of the Euler totient function of terms in a Lucas sequence giving a factorial, a power of $2$, or a product of a power of $2$ and a power of $3$, see \cite{damir}, \cite{luca5}, \cite{luca7}, and \cite{saunders}. Here we prove the following.
\begin{theorem}\label{thm}
Let $P\geq 3$ be a natural number and take the Lucas sequence of the first kind that is $u_0=0$, $u_1=1$, and $u_n=Pu_{n-1}+u_{n-2}$ for all $n\geq 2$. Let $\varphi$ denote the Euler totient function. Then the only solution to the Diophantine equation $\varphi\left(u_n\right)=u_m$ is $\varphi\left(u_1\right)=\varphi(1)=1=u_1$.
\end{theorem}
\begin{remark}
\normalfont
Theorem \ref{thm} combined with the results in \cite{luca2} and \cite{faye2} completely exhausts the problem of finding solutions to $\varphi\left(u_n\right)=u_m$ where $\left(u_n\right)_n$ is a Lucas sequence of the first kind with recurrence relation $u_n=Pu_{n-1}+u_{n-2}$ and $P>0$. If, instead, we have $P<0$ and $\left(u_n\right)_n$ is the corresponding Lucas sequence, then we can see that $\left(\left|u_n\right|\right)_n$ is the corresponding Lucas sequence of $-P$. Therefore, it suffices to investigate the case of $P>0$ for a Lucas sequence of the first kind with $Q=1$, which we do here.
\end{remark}
To prove Theorem \ref{thm}, we begin by following the proofs in \cite{luca2} and \cite{faye2}, but we also borrow proof techniques from \cite{chen} and \cite{faye2} too. Throughout this paper, we let $\left(u_n\right)_n$ denote the Lucas sequence of the first kind in question and $\left(v_n\right)_n$ denote the corresponding Lucas sequence of the second kind in question as well. First, we prove some preliminary lemmas. Then we prove that $n>m\geq 2^{416}$ and $k\geq 416$, where $k$ is the number of distinct prime factors of $u_n$. Next we obtain an upper bound on $l:=n-m$ in terms of $n$ and $P$. Using this bound we bound the prime factors of $u_n$, giving us an upper bound on $n$ in terms of $k$ and $P$, which helps us to quickly prove that $n$ cannot be odd. Then we prove that $l\leq 12$. The two final sections are devoted to the cases of $P\geq 80$ and $3\leq P\leq 79$.
\section{Preliminary Lemmas}
We first deal with some necessary preliminary lemmas, which stem from well-known results on Lucas sequences in general. Let $D:=P^2+4$, $\alpha:=\frac{P+\sqrt{D}}{2}$ and $\beta:=\frac{P-\sqrt{D}}{2}$. We have the Binet formula for $\left(u_n\right)_n$, which is
\begin{equation*}
u_n=\frac{\alpha^n-\beta^n}{\alpha-\beta},
\end{equation*}
valid for all $n\in\mathbb{N}$. This implies that
\begin{equation*}
\alpha^{n-2}<u_n<\alpha^{n-1}
\end{equation*}
holds for all $n\in\mathbb{N}$. Also, $P\geq 3$ implies that $\log\alpha>1.19$. We will be using this inequality throughout the paper. The following lemmas on the Lucas sequences $\left(u_n\right)_n$ and $\left(v_n\right)_n$ are easy to prove and found in \cite{lehmer}.
\begin{lemma}\label{lem1}
For all $n\geq 0$, we have $u_{2n}=u_nv_n$ and $v_n^2-Du_n^2=4(-1)^n$.
\end{lemma}
\begin{lemma}\label{lem3}
1) If $m,n\in\mathbb{N}$ such that $m\mid n$, then $u_m\mid u_n$.
\newline
\newline
2) For all $m,n\in\mathbb{N}$, we have $\gcd\left(u_m,u_n\right)=u_{\gcd(m,n)}$.
\end{lemma}
For the next results, we need some notation.
\begin{notation}
For a prime $p$ and a number $n$, we write $p^k\|n$ if $p^k$ is the highest power of $p$ dividing $n$.
\end{notation}
\begin{lemma}\label{lem2}
Let $a,k,m\in\mathbb{N}$ and $q$ be a prime such that $q^a\|u_m$ and $q\nmid k$. Then for any $l\geq 0$, we have $q^{a+l}|u_{kmq^l}$ with $q^{a+l}\|u_{kmq^l}$ if $q^a\neq 2$.
\end{lemma}
We also have the following special case of $q=2$ of Lemma \ref{lem2}.
\begin{lemma}\label{lem9}
Let $m,j\in\mathbb{N}$.
\newline
\newline
i) Suppose $P$ is odd. Then $3\cdot 2^j\mid m$ if and only if $2^{j+2}\mid u_m$. 
\newline
\newline
ii) Suppose $P$ is even and $2^{t_1}\|P$. Then $2^j\mid m$ if and only if $2^{j+t_1-1}\mid u_m$.
\end{lemma}
Since we will be using Lemma \ref{lem9} throughout this paper, we will be setting $t_1:=\nu_2\left(P\right)$ regardless of whether $P$ is odd or even, i.e. $t_1$ is the integer such that $2^{t_1}$ is the highest power of $2$ dividing $P$. Also note that Lemmas \ref{lem3} and \ref{lem2} imply that for all natural numbers $N$ there exists $z(N)\in\mathbb{N}$ such that $N\mid u_n$ if and only if $z(N)\mid n$. We call $z(N)$ the \textit{order of appearance} of $N$ in the sequence $\left(u_n\right)_n$. In particular, this is true if $N=p$ is a prime and, in fact, we have the following result, due to Lucas \cite{lucas}.
\begin{lemma}{(Lucas (1878),\cite{lucas})}\label{lem18}
Let $p$ be a prime. Then $z(p)=p$ if $p\mid D$. Also, if $p\nmid D$ and $D$ is a quadratic residue $\pmod p$, then $z(p)\mid p-1$. If $p\nmid D$ and $D$ is not a quadratic residue $\pmod p$, then $z(p)\mid p+1$.   
\end{lemma}
Another interesting result we will be using on the prime factors of $u_n$ is due to Carmichael. First, we have the following definition.
\begin{definition}
A \textit{primitive prime factor} of $u_n$ (respectively, $v_n$) is a prime factor $p$ of $u_n$ (respectively, $v_n$) such that $p\nmid u_m$ (respectively, $p\nmid v_m$) for all $1\leq m<n$. 
\end{definition}
\begin{lemma}{(Carmichael (1913),\cite{carmichael})}\label{lem15}
If $n\neq 1,2,6$, then $u_n$ has a \textit{primitive prime factor}, except in the case of $n=12$ in the usual Fibonacci sequence $1,1,2,3,5,8,\ldots$.
\end{lemma}
We can also see from Lemma \ref{lem2} that if $p^{e_p}$ is the highest power of $p$ that divides $u_{z(p)}$ for $p\geq 3$, then for any $\lambda\geq 0$, we have that $p^{e_p+\lambda}\mid u_n$ if and only if $z(p)p^{\lambda}\mid n$. From Lemma \ref{lem18} and this fact, we can also derive the following lemma.
\begin{lemma}\label{lem4}
For all odd primes $p$ such that $p\nmid P^2+4$, we have
\begin{equation*}
e_p\leq\frac{(p+1)\log\alpha+0.2}{2\log p}.
\end{equation*}
\end{lemma}
\begin{proof}
Since $2<p$ and $p\nmid P^2+4$, we have $z(p)\mid p+\epsilon$ where $\epsilon=\pm 1$. By Lemma \ref{lem1}, we have
\begin{equation*}
p^{e_p}\mid u_{p+\epsilon}=u_{(p+\epsilon)/2}v_{(p+\epsilon)/2}.
\end{equation*}
Since $p$ is odd, it cannot divide both $u_{(p+\epsilon)/2}$ and $v_{(p+\epsilon)/2}$ by Lemma \ref{lem1}. So either $p^{e_p}\mid u_{(p+\epsilon)/2}$ or $p^{e_p}\mid v_{(p+\epsilon)/2}$. In the former case, we have
\begin{equation*}
p^{e_p}\leq u_{(p+\epsilon)/2}<\alpha^{(p+\epsilon)/2}
\end{equation*}
and in the latter case, we have
\begin{equation*}
p^{e_p}\leq v_{(p+\epsilon)/2}<\alpha^{(p+\epsilon)/2}+1.
\end{equation*}
So, in either case, we have $p^{e_p}<\alpha^{(p+1)/2}+1$. Then we have
\begin{align*}
e_p\log p&<\log\left(\alpha^{(p+1)/2}+1\right)\\
&<\log\left(\alpha^{(p+1)/2}\right)+\log\left(1+\frac{1}{\alpha^{(p+1)/2}}\right)\\
&\leq\log\left(\alpha^{(p+1)/2}\right)+\log\left(1+\alpha^{-2}\right)\\
&<\frac{(p+1)\log(\alpha)}{2}+0.1,
\end{align*}
from which the result follows.
\end{proof}
We also require a result from prime number theory. We denote by $\omega(n)$ the number of distinct prime factors of $n$. Parts 1), 2), and 3) of the following lemma are (3.13), (3.29), and (3.41) in \cite{rosser} respectively, while the last part is Th\'{e}or\'{e}me 13 in \cite{robin}.
\begin{lemma}\label{lem5}
Let $p_1<p_2<\ldots$ be the sequence of prime numbers. We have the following.
\newline
\newline
1) For all $n\geq 6$, we have $p_n<n\left(\log n+\log\log n\right)$.
\newline
\newline
2) For all $x\geq 2861$, we have
\begin{equation*}
\prod_{p\leq x}\left(1+\frac{1}{p-1}\right)<1.79\log x\left(1+\frac{1}{2(\log x)^2}\right).
\end{equation*}
\newline
\newline
3) For all $n\geq 3$, we have
\begin{equation*}
\varphi(n)>\frac{n}{1.79\log\log n+\displaystyle\frac{2.5}{\log\log n}}.
\end{equation*}
\newline
\newline
4) For all $n\geq 26$, we have
\begin{equation*}
\omega(n)<\frac{\log n}{\log\log n-1.1714}.
\end{equation*}
\end{lemma}
Throughout the paper, we will let $\mathcal{P}_n=\{p:z(p)=n\}$, $\mathcal{P}'_n=\{p>\alpha^4:z(p)=n\}$, 
\begin{equation*}
S_n:=\sum_{p\in\mathcal{P}_n}\frac{1}{p},
\end{equation*}
and
\begin{equation*}
T_n:=\sum_{p\in\mathcal{P}'_n}\frac{1}{p}.
\end{equation*}
We have the following two lemmas involving $T_n$ and $S_n$.
\begin{lemma}\label{lem6}
If $n>\alpha^2$, then
\begin{equation*}
T_n<\frac{2\log n}{n}.
\end{equation*}
\end{lemma}
\begin{proof}
First, suppose that $n>\alpha^4$. For any $p\in \mathcal{P}'_n$, we have $n\mid p-1$ or $n\mid p+1$ since $p>\alpha^4>D$ and $n>\alpha^4>2$. Letting $l_n$ be the size of the set $\mathcal{P}'_n$, we have
\begin{equation*}
(n-1)^{l_n}\leq\prod_{p\in\mathcal{P}'_n}p\leq u_n<\alpha^{n-1},
\end{equation*}
so that
\begin{equation}\label{eqn35}
l_n<\frac{(n-1)\log\alpha}{\log(n-1)}<\frac{n\log\alpha}{\log n}.
\end{equation}
Therefore, we have the following:
\begin{align*}
T_n&\leq 2\sum_{1\leq l\leq l_n}\frac{1}{nl-1}\\
&=\frac{2}{n}\sum_{1\leq l\leq l_n}\frac{1}{l-\displaystyle\frac{1}{n}}\\
&\leq\frac{2}{n}\left(\int_{1-\frac{1}{n}}^{l_n-\frac{1}{n}}\frac{dx}{x}+\frac{1}{1-\displaystyle\frac{1}{n}}\right)\\
&=\frac{2}{n}\left(\log\left(l_n-\frac{1}{n}\right)-\log\left(1-\frac{1}{n}\right)+\frac{n}{n-1}\right)\\
&<\frac{2}{n}\left(\log\left(\frac{l_n}{1-\displaystyle\frac{1}{n}}\right)+\frac{n}{n-1}\right)\\
&<\frac{2}{n}\log\left(\frac{n(\log\alpha) e^{\frac{n}{n-1}}}{\left(1-\displaystyle\frac{1}{n}\right)\log n}\right)\\
&<\frac{2}{n}\log\left(\frac{n e^{\frac{n}{n-1}}}{\left(4-\displaystyle\frac{4}{n}\right)}\right).
\end{align*}
Since $n>\alpha^4>119$, we have
\begin{equation*}
e^{\frac{n}{n-1}}<e^{\frac{119}{118}}<4-\frac{4}{119}<4-\frac{4}{n}.
\end{equation*}
The result follows for $n>\alpha^4$. Now assume that $\alpha^2<n\leq \alpha^4$. From \eqref{eqn35}, we have $l_n<\frac{n}{4}$. Therefore, we deduce the following:
\begin{align*}
T_n&\leq 2\sum_{0\leq l\leq l_n-1}\frac{1}{\alpha^4+nl}\\
&\leq\frac{2}{n}\left(\int_0^{\frac{n}{4}-1}\frac{dx}{\displaystyle\frac{\alpha^4}{n}+x}+\frac{n}{\alpha^4}\right)\\
&=\frac{2}{n}\left(\log\left(\frac{\alpha^4}{n}+\frac{n}{4}-1\right)-\log\left(\frac{\alpha^4}{n}\right)\right)+\frac{2}{\alpha^4}\\
&\leq\frac{2}{n}\left(\log\left(1+\frac{n^2}{4\alpha^4}-\frac{n}{\alpha^4}\right)+1\right)\\
&=\frac{2}{n}\log\left(e+\frac{n^2e}{4\alpha^4}-\frac{ne}{\alpha^4}\right)
\end{align*}
Since $n>\alpha^2>4$, we have
\begin{equation*}
e+\frac{n^2e}{4\alpha^4}-\frac{ne}{\alpha^4}=e+\frac{n}{\alpha^4}\left(\frac{en}{4}-e\right)<e+\frac{en}{4}-e<n.
\end{equation*}
The result follows.
\end{proof}
\begin{lemma}\label{lem16}
If $n\geq 40$, then
\begin{equation*}
S_n<\frac{3.204}{n}+\frac{1}{n\log n}+\frac{4\log\log n}{\varphi(n)}+\frac{4\log\log\alpha}{\varphi(n)\log n}.
\end{equation*}
\end{lemma}
\begin{proof}
We follow the proofs in \cite{chen} and \cite{faye3}. For convenience to the reader, we give the details here. As in \eqref{eqn35}, we derive
\begin{equation*}
\left|\mathcal{P}_n\right|\leq\frac{n\log\alpha}{\log n}. 
\end{equation*}
Let $\mathcal{A}_n:=\left\{p\leq 4n:n\mid p-1, n\mid p+1,\text{ or }n=p\right\}$. We divide the sum up in the definition of $S_n$ as follows:
\begin{align*}
S_n&=\sum_{\substack{p\leq 4n\\z(p)=n}}\frac{1}{p}+\sum_{\substack{4n<p\leq n^2\log\alpha\\z(p)=n}}\frac{1}{p}+\sum_{\substack{p>n^2\log\alpha\\z(p)=n}}\frac{1}{p}\\
&\leq\sum_{p\in\mathcal{A}_n}\frac{1}{p}+\sum_{\substack{4n<p\leq n^2\log\alpha\\n\mid p-1\text{ or }n\mid p+1}}\frac{1}{p}+\sum_{\substack{p>n^2\log\alpha\\p\in\mathcal{P}_n}}\frac{1}{p}\\
&:=Y_1+Y_2+Y_3.
\end{align*}
We first estimate the sum $Y_2$. Let $\pi(X;n,1)$ and $\pi(X;n,n-1)$ denote the number of primes $p\leq X$ such that $n\mid p-1$ or $n\mid p+1$, respectively. By the Brun-Titchmarsh Theorem from Montgomery and Vaughan \cite{montgomery}, we have
\begin{equation*}
\pi(X;n,1),\pi(X;n,n-1)<\frac{2X}{\varphi(n)\log\left(X/n\right)},
\end{equation*}
valid for all $X>n\geq 2$. Therefore, by Abel Summation, we derive the following:
\begin{align*}
Y_2&\leq\int_{4n}^{n^2\log\alpha}\frac{d\pi(x;n,1)}{x}+\int_{4n}^{n^2\log\alpha}\frac{d\pi(x;n,n-1)}{x}\\
&=\frac{\pi(x;n,1)+\pi(x;n,n-1)}{x}\bigg|_{x=4n}^{n^2\log\alpha}+\int_{4n}^{n^2\log\alpha}\frac{\pi(x;n,1)+\pi(x;n,n-1)}{x^2}dx\\
&\leq\frac{4}{\varphi(n)\log(n\log\alpha)}-\frac{\pi(4n;n,1)+\pi(4n;n,n-1)}{4n}+\frac{4}{\varphi(n)}\int_{4n}^{n^2\log\alpha}\frac{dx}{x\log\left(x/n\right)}\\
&=\frac{4\log\log(n\log\alpha)}{\varphi(n)}-\frac{\pi(4n;n,1)+\pi(4n;n,n-1)}{4n}+\frac{4}{\varphi(n)}\left(\frac{1}{\log(n\log\alpha)}-\log\log 4\right).
\end{align*}
Since $n\geq 40$ and $\alpha>3$, we have
\begin{equation*}
\frac{1}{n\log\alpha}<\frac{1}{\log(40)}<\log\log 4,
\end{equation*}
so that
\begin{equation*}
Y_1+Y_2\leq\frac{4\log\log(n\log\alpha)}{\varphi(n)}-\frac{\pi(4n;n,1)+\pi(4n;n,n-1)}{4n}+\sum_{p\in\mathcal{A}_n}\frac{1}{p}.
\end{equation*}
Notice that
\begin{align*}
&\quad-\frac{\pi(4n;n,1)+\pi(4n;n,n-1)}{4n}+\sum_{p\in\mathcal{A}_n}\frac{1}{p}\\
&\leq\frac{1}{n-1}+\frac{1}{n}+\frac{1}{n+1}+\frac{1}{2n-1}+\frac{1}{2n+1}+\frac{1}{3n-1}+\frac{1}{3n+1}+\frac{1}{4n-1}-\frac{7}{4n}\\
&\leq\frac{1}{n}\left(\frac{n}{n-1}+2+\frac{n}{2n-1}+\frac{1}{2}+\frac{n}{3n-1}+\frac{1}{3}+\frac{n}{4n-1}-\frac{7}{4}\right)\\
&<\frac{3.204}{n}.
\end{align*}
Thus,
\begin{equation*}
Y_1+Y_2\leq\frac{3.204}{n}+\frac{4\log\log(n\log\alpha)}{\varphi(n)}.
\end{equation*}
For $Y_3$, we have
\begin{equation*}
Y_3<\frac{\left|\mathcal{P}_n\right|}{n^2\log\alpha}\leq\frac{1}{n\log n}.
\end{equation*}
Thus,
\begin{equation*}
S_n<\frac{3.204}{n}+\frac{1}{n\log n}+\frac{4\log\log(n\log\alpha)}{\varphi(n)}.
\end{equation*}
We deduce
\begin{align*}
\log\log(n\log\alpha)&=\log(\log n+\log\log\alpha)\\
&=\log\log n+\log\left(1+\frac{\log\log\alpha}{\log n}\right)\\
&<\log\log n+\frac{\log\log\alpha}{\log n}.
\end{align*}
Thus,
\begin{equation*}
S_n<\frac{3.204}{n}+\frac{1}{n\log n}+\frac{4\log\log n}{\varphi(n)}+\frac{4\log\log\alpha}{\varphi(n)\log n}.
\end{equation*}
\end{proof}
We also have the following result from \cite{luca6}.
\begin{lemma}\label{lem8}
We have
\begin{equation*}
\sum_{d\mid n}\frac{\log d}{d}<\left(\sum_{p\mid n}\frac{\log p}{p-1}\right)\frac{n}{\varphi(n)}.
\end{equation*}
\end{lemma}
\begin{lemma}\label{lem}
If $\varphi\left(u_n\right)=u_m$ with $n>1$, then $m$ is even.
\end{lemma}
\begin{proof}
Suppose that $\varphi\left(u_n\right)=u_m$ holds for some $n>m\geq 1$. Then we must have $m>1$. If $u_m$ is odd, then we must have $u_m=1$, but then $u_m\geq u_2=P\geq 3$. Therefore, $u_m$ is even. If $P$ is even, then Lemma \ref{lem9} implies that $u_m$ is even if and only if $m$ is even. Therefore, $m$ is even if $P$ is even. Now assume $P$ is odd. Again, Lemma \ref{lem9} implies that $u_m$ is even if and only if $3\mid m$. Therefore, $3\mid m$. Suppose for a contradiction that $m$ is odd. Then $m\equiv 3\pmod 6$. Analyzing the sequence $\left(u_m\right)_m$ modulo $4$ reveals that $u_6\equiv 0\pmod 4$ and $u_7\equiv 1\pmod 4$. We deduce that $u_m\equiv u_3\pmod 4$. Since $u_3=P^2+1$ and $P$ is odd, we have $u_3\equiv 2\pmod 4$, so that $u_m\equiv 2\pmod 4$. Thus, $u_n$ can have at most one odd prime factor and one of three cases must hold: $u_n=4$, $u_n=q^{\gamma}$, or $u_n=2q^{\gamma}$ where $\gamma\in\mathbb{N}$ and $q$ is an odd prime. If $u_n=4$, then $u_m=2$, but then $u_m\geq u_2\geq 3$, eliminating that case. If $u_n=q^{\gamma}$, then $u_m=(q-1)q^{\gamma-1}$, so that
\begin{equation*}
\frac{u_{n-1}}{u_n}\geq\frac{u_m}{u_n}=1-\frac{1}{q}\geq 1-\frac{1}{3}=\frac{2}{3}.
\end{equation*}
Thus $2u_n\leq 3u_{n-1}<u_n$, which can't hold, eliminating that case. If $u_n=2q^{\gamma}$, then we have $u_m=(q-1)q^{\gamma-1}$ and
\begin{equation*}
\frac{u_{n-1}}{u_n}\geq\frac{u_m}{u_n}=\frac{1}{2}-\frac{1}{2q}\geq\frac{1}{2}-\frac{1}{6}=\frac{1}{3},
\end{equation*}
so that $u_n\leq 3u_{n-1}<u_n$, again, a contradiction, completing the proof.
\end{proof}
Let $u_n=q_1^{\alpha_1}q_2^{\alpha_2}\cdots q_k^{\alpha_k}$ be the prime factorisation of $u_n$. Let $l:=n-m$. Since $m$ is even by Lemma \ref{lem}, we can see that $\beta^n<\alpha^l\beta^m$, so that
\begin{equation}\label{eqn1}
\alpha^l<\frac{\alpha^n-\beta^n}{\alpha^m-\beta^m}=\frac{u_n}{u_m}=\frac{u_n}{\varphi\left(u_n\right)}.
\end{equation}
We make use of \eqref{eqn1} repeatedly throughout the paper.
\section{The Case of $m<2^{416}$}
In this section, we completely eliminate the possibility that $m<2^{416}$. We divide into two cases: $n$ is odd and $n$ is even.
\begin{case}{$n$ is odd}
\newline
\newline
By Lemma \ref{lem1}, we have
\begin{equation*}
v_n^2\equiv -4\pmod{q_i}
\end{equation*}
for all $1\leq i\leq k$. It follows that $-1$ is a quadratic residue $\pmod{q_i}$, so that $q_i\equiv 1\pmod 4$ if $q_i$ is odd. Let $p_i'$ be the $ith$ prime number that is congruent to $1\pmod 4$. By Maple, we can calculate
\begin{equation}\label{eqn31}
2\left(\frac{p_{415}'}{p_{415}'-1}\right)^t\prod_{i=1}^{414}\left(\frac{p_i'}{p_i'-1}\right)<4.592\cdot 1.0002^t<\max\left\{5,2^t\right\}
\end{equation}
for all $t\geq 0$. Assume that $P\geq 5$. Then we have
\begin{equation}\label{eqn32}
\max\left\{5,2^{t_1}\right\}\leq P<\frac{u_n}{\varphi\left(u_n\right)}=\prod_{i=1}^k\left(\frac{q_i}{q_i-1}\right).
\end{equation}
Combining \eqref{eqn31} and \eqref{eqn32}, we can see that $k\geq 416+t_1$. Therefore, $2\cdot 4^{415+t_1}\mid u_m$, or $2^{831+2t_1}\mid u_m$. Lemma \ref{lem9} gives $2^{829}\mid m$, contradicting $m<2^{416}$. Suppose $P=4$. Since $n$ is odd, all of the $q_i$'s are odd, so that $q_i\equiv 1\pmod 4$ for all $1\leq i\leq k$. Again, by Maple, we calculate
\begin{equation*}
\prod_{i=1}^{415}\left(\frac{p_i'}{p_i'-1}\right)<2.296<4<\frac{u_n}{\varphi\left(u_n\right)}=\prod_{i=1}^k\left(\frac{q_i}{q_i-1}\right).
\end{equation*}
Therefore, $k\geq 416$, so that $2^{832}\mid u_m$. Lemma \ref{lem9} gives $2^{830}\mid m$, contradicting $m<2^{416}$. Finally, suppose $P=3$. Since $u_n\geq 3$, $u_m$ must be even. So $3\mid m$ by Lemma \ref{lem2}. If $u_n$ is odd, then we use the exact same argument in the $P=4$ case to conclude that $m<2^{416}$ is impossible. So assume $u_n$ is even. Then $3\mid n$ by Lemma \ref{lem2} and $l\geq 3$. Therefore,
\begin{equation*}
2\prod_{i=1}^{414}\left(\frac{p_i'}{p_i'-1}\right)<4.592<27<\frac{u_n}{\varphi\left(u_n\right)}=\prod_{i=1}^k\left(\frac{q_i}{q_i-1}\right),
\end{equation*}
so that $k\geq 416$. Then $2\cdot 4^{415}\mid u_m$, so that $2^{829}\mid m$, again contradicting $m<2^{416}$.
\end{case}
\begin{case}{$n$ is even}
\newline
\newline
Let $p_i$ denote the $i$th prime number. Suppose that $P\geq 4$. Since $m$ is even, we have $l\geq 2$, so that
\begin{equation}\label{eqn33}
\max\left\{16,2^{t_1}\right\}\leq P^2<\frac{u_n}{\varphi\left(u_n\right)}=\prod_{i=1}^k\left(\frac{q_i}{q_i-1}\right).
\end{equation}
By Maple, we can calculate
\begin{equation}\label{eqn34}
\left(\frac{p_{420}}{p_{420}-1}\right)^t\prod_{i=1}^{419}\left(\frac{p_i}{p_i-1}\right)<14.24\cdot 1.0004^t<\max\left\{16,2^t\right\}
\end{equation}
for any $t\geq 0$. Combining \eqref{eqn33} and \eqref{eqn34}, we can see that $k\geq 419+t_1$. Thus $2^{418+t_1}\mid u_m$, so that $2^{416}\mid m$, contradicting $m<2^{416}$. Suppose that $P=3$. Again, as in the case of $n$ being odd, we have $3\mid m$. Suppose that $u_n$ is odd. By Maple, we have
\begin{equation*}
\prod_{i=2}^{419}\left(\frac{p_i}{p_i-1}\right)<7.12<9<\frac{u_n}{\varphi\left(u_n\right)}=\prod_{i=1}^k\left(\frac{q_i}{q_i-1}\right).
\end{equation*}
So $k\geq 419$. Thus $2^{418}\mid u_m$, so that $2^{416}\mid m$, again contradicting $m<2^{416}$. So assume that $u_n$ is even. Then $3\mid n$ and, again, $l\geq 3$. Therefore,
\begin{equation*}
\prod_{i=1}^{419}\left(\frac{p_i}{p_i-1}\right)<14.24<27<\frac{u_n}{\varphi\left(u_n\right)}=\prod_{i=1}^k\left(\frac{q_i}{q_i-1}\right),
\end{equation*}
so, again, $k\geq 419$, leading to the impossibility of $m<2^{416}$.
\end{case}
We have proved the following:
\begin{lemma}
If $\varphi\left(u_n\right)=u_m$ and $n>1$, then $n>m\geq 2^{416}$ and $k\geq 416$.
\end{lemma}
\section{Bounding $l$ in terms of $n$ and $\alpha$}
Our next task is to bound $l$ in terms of $n$ and $\alpha$. We divide into two cases: $P$ is odd and $P$ is even.
\setcounter{case}{0}
\begin{case}{$P$ is odd}
\newline
\newline
Let $k(n):=\frac{\log n}{\log 2}+2=\frac{\log(4n)}{\log 2}$. Clearly, $2^k\mid u_m$, so that from Lemma \ref{lem9}, we have $2^{k-2}\mid m$, so that $2^k<4n$. Thus, $k<k(n)$. Letting $p_j$ be the $j$th prime number, we therefore deduce from Lemma \ref{lem5} that
\begin{equation*}
p_k\leq p_{\lfloor k(n)\rfloor}<k(n)\left(\log k(n)+\log\log k(n)\right):=q(n).
\end{equation*}
Since $n>2^{416}$, $k(n)\geq 418$, so that $q(n)>3274$ and part 2) of Lemma \ref{lem5} applies with $x=q(n)$ to which we obtain
\begin{equation*}
\frac{\varphi\left(u_n\right)}{u_n}=\prod_{i=1}^k\left(1-\frac{1}{q_i}\right)\geq\prod_{2\leq p\leq q(n)}\left(1-\frac{1}{p}\right)>\frac{1}{1.79\log q(n)\left(1+\displaystyle\frac{1}{2(\log q(n))^2}\right)}.
\end{equation*}
Thus, from \eqref{eqn1}, we have
\begin{equation*}
\alpha^l<1.79\log q(n)\left(1+\frac{1}{2(\log q(n))^2}\right).
\end{equation*}
Taking logarithms, we get
\begin{align}
l\log\alpha &<\log(1.79)+\log\left(1+\frac{1}{2(\log 3274)^2}\right)+\log\log q(n)\nonumber\\
&<0.59+\log\log q(n)\label{eqn2}.
\end{align}
We also have that in our range of $n$ that $q(n)<(\log(4n))^{1.45}$, so that, by \eqref{eqn2}, we have
\begin{equation}\label{eqn3}
l\log\alpha<0.59+\log 1.45+\log\log\log(4n)=0.59+\log 1.45+\log\log(\log 4+\log n).
\end{equation}
Note that
\begin{equation*}
\log 4+\log n=\log n\left(1+\frac{\log 4}{\log n}\right)<\left(1+\frac{\log 4}{216\log 2}\right)\log n<\frac{109\log n}{108},
\end{equation*}
so that, from \eqref{eqn3}, we have
\begin{align*}
l\log\alpha &<0.59+\log 1.45+\log\log\left(\frac{109\log n}{108}\right)\\
&=0.59+\log 1.45+\log\left(\log\left(\frac{109}{108}\right)+\log\log n\right).
\end{align*}
Also, note that
\begin{align*}
\log\left(\frac{109}{108}\right)+\log\log n&=\log\log n\left(1+\frac{\log\left(\displaystyle\frac{109}{108}\right)}{\log\log n}\right)\\
&<\left(1+\frac{\log\left(\displaystyle\frac{109}{108}\right)}{\log\log\left(2^{416}\right)}\right)\log\log n\\
&<1.002\log\log n,
\end{align*}
so that
\begin{align*}
l\log\alpha&<0.59+\log 1.45+\log(1.002\log\log n)\\
&=0.59+\log 1.45+\log 1.002+\log\log\log n\\
&<0.97+\log\log\log n.
\end{align*}
Thus,
\begin{equation*}
l<\frac{0.97+\log\log\log n}{\log\alpha}<\frac{\log\log\log n}{\log\alpha}+0.82,
\end{equation*}
where we have used the fact that $\log\alpha>1.19$.
\end{case}
\begin{case}{$P$ is even}
\newline
\newline
Let $k(n):=\frac{\log n}{\log 2}+t_1-1=\frac{\log\left(2^{t_1-1}n\right)}{\log 2}$. Clearly, $2^k\mid u_m$, so that from Lemma \ref{lem9}, we have $2^{k-t_1+1}\mid m$, so that $2^k<2^{t_1-1}n$. Thus, $k<k(n)$. Letting $p_j$ be the $j$th prime number, we therefore deduce from Lemma \ref{lem5} that
\begin{equation*}
p_k\leq p_{\lfloor k(n)\rfloor}<k(n)\left(\log k(n)+\log\log k(n)\right):=q(n).
\end{equation*}
Since $n>2^{416}$, $k(n)\geq 416$, so that $q(n)\geq 3256$, and part 2) of Lemma \ref{lem5} applies with $x=q(n)$. We obtain
\begin{equation*}
\frac{\varphi\left(u_n\right)}{u_n}=\prod_{i=1}^k\left(1-\frac{1}{q_i}\right)\geq\prod_{2\leq p\leq q(n)}\left(1-\frac{1}{p}\right)>\frac{1}{1.79\log q(n)\left(1+\displaystyle\frac{1}{2(\log q(n))^2}\right)}.
\end{equation*}
Thus, from \eqref{eqn1}, we have
\begin{equation*}
\alpha^l<1.79\log q(n)\left(1+\frac{1}{2(\log q(n))^2}\right).
\end{equation*}
Taking logarithms, we get
\begin{align}
l\log\alpha &<\log(1.79)+\log\left(1+\frac{1}{2(\log 3256)^2}\right)+\log\log q(n)\nonumber\\
&<0.59+\log\log q(n)\label{eqn4}.
\end{align}
We also have that in our range of $n$ that $q(n)<\left(\left(\log(2^{t_1-1}n\right))\right)^{1.45}$, so that, by \eqref{eqn4}, we have
\begin{align}
l\log\alpha &<0.59+\log 1.45+\log\log\log\left(2^{t_1-1}n\right)\nonumber\\
&=0.59+\log 1.45+\log\log\left(\left(t_1-1\right)\log 2+\log n\right)\label{eqn5}.
\end{align}
Note that
\begin{equation*}
\left(t_1-1\right)\log 2+\log n=\left(1+\frac{\left(t_1-1\right)\log 2}{\log n}\right)\log n,
\end{equation*}
so that, from \eqref{eqn5}, we have
\begin{align}
l\log\alpha &<0.59+\log 1.45+\log\log\left(\left(1+\frac{\left(t_1-1\right)\log 2}{\log n}\right)\log n\right)\nonumber\\
&=0.59+\log 1.45+\log\left(\log\left(1+\frac{\left(t_1-1\right)\log 2}{\log n}\right)+\log\log n\right)\nonumber\\
&<0.59+\log 1.45+\log\left(\frac{\left(t_1-1\right)\log 2}{\log n}+\log\log n\right)\label{eqn6}.
\end{align}
Note that
\begin{align*}
\frac{\left(t_1-1\right)\log 2}{\log n}+\log\log n&<\log\log n\left(1+\frac{\left(t_1-1\right)\log 2}{\log\left(2^{416}\right)\log\log\left(2^{416}\right)}\right)\\
&<\left(1+0.0005\left(t_1-1\right)\right)\log\log n,
\end{align*}
so that, from \eqref{eqn6}, we have
\begin{align*}
l\log\alpha &<0.59+\log 1.45+\log\left(1+0.0005\left(t_1-1\right)\right)+\log\log\log n\\
&<0.59+\log 1.45+0.0005\left(t_1-1\right)+\log\log\log n\\
&<0.97+0.0005\left(t_1-1\right)+\log\log\log n.
\end{align*}
Thus,
\begin{equation*}
l<0.82+\frac{0.0005\left(t_1-1\right)}{\log\alpha}+\frac{\log\log\log n}{\log\alpha}.
\end{equation*}
Since $\log\alpha>\log P\geq \log\left(2^{t_1}\right)=t_1\log 2$, we deduce
\begin{equation*}
l<0.82+\frac{0.0005}{\log 2}+\frac{\log\log\log n}{\log\alpha}<0.83+\frac{\log\log\log n}{\log\alpha}.
\end{equation*}
\end{case}
Thus, regardless of the value of $P\geq 3$, we have proved the following.
\begin{lemma}\label{lem10}
We have $n>m\geq 2^{416}$ and
\begin{equation*}
l<\frac{\log\log\log n}{\log\alpha}+0.83.
\end{equation*}
\end{lemma}
\section{Bounding the primes $q_i$ for all $1\leq i\leq k$}
We next bound all the primes dividing $u_n$. Note that $u_m=\prod_{i=1}^k\left(q_i-1\right)q_i^{\alpha_i-1}$. Let $B=\prod_{i=1}^kq_i^{\alpha_i-1}$, so that $u_n=q_1\cdots q_kB$. Also, $B\mid\gcd\left(u_m,u_n\right)$. By Lemma \ref{lem3}, we have $B\mid u_{\gcd(m,n)}$, so $B\mid u_l$. Thus,
\begin{equation}\label{eqn7}
B\leq u_l<\alpha^{l-1}<(\log\log n)\alpha^{-0.17}.
\end{equation}
We use the argument in Section 3.3 of \cite{luca2} to bound the primes. Note that
\begin{equation*}
1-\prod_{i=1}^k\left(1-\frac{1}{q_i}\right)=1-\frac{u_m}{u_n}=\frac{u_n-u_m}{u_n}\geq\frac{u_n-u_{n-1}}{u_n}>\frac{(P-1)u_{n-1}}{u_n},
\end{equation*}
so that
\begin{equation*}
\frac{(P-1)u_{n-1}}{u_n}<\sum_{i=1}^k\frac{1}{q_i}<\frac{k}{q_1}.
\end{equation*}
Thus,
\begin{equation*}
q_1<\frac{ku_n}{(P-1)u_{n-1}}<\frac{k(P+1)}{(P-1)}\leq 2k.
\end{equation*}
Defining
\begin{equation*}
y_i:=\prod_{j=1}^iq_j,
\end{equation*}
we can deduce that
\begin{equation*}
y_i<\left(2k\alpha^{1.83}\log\log n\right)^{\frac{3^i-1}{2}}
\end{equation*}
by induction on $i$, using the same method to prove inequality (13) in \cite{luca2}. In particular,
\begin{equation*}
q_1\cdots q_k<\left(2k\alpha^{1.83}\log\log n\right)^{\frac{3^k-1}{2}}.
\end{equation*}
Using Lemma \ref{lem10} and \eqref{eqn7}, we deduce
\begin{equation*}
u_n=q_1\cdots q_kB<\left(2k\alpha^{1.83}\log\log n\right)^{\frac{3^k-1}{2}}(\log\log n)\alpha^{-0.17}<\left(2k\alpha^{1.83}\log\log n\right)^{\frac{3^k+1}{2}}.
\end{equation*}
Since $\alpha^{n-2}<u_n$, we deduce
\begin{equation*}
(n-2)\log\alpha<\frac{\left(3^k+1\right)}{2}\log\left(2k\alpha^{1.83}\log\log n\right).
\end{equation*}
Thus,
\begin{equation}\label{eqn8}
3^k>(n-2)\left(\frac{2\log\alpha}{\log\left(2k\alpha^{1.83}\log\log n\right)}\right)-1.
\end{equation}
\section{The Case when $n$ is odd}
Here we show that if $n>1$ is odd, then $\varphi\left(u_n\right)=u_m$ has no solutions. We alredy know that $n>m>2^{416}$. We assume for a contradiction that $n$ and $m$ satisfy these conditions. We know that $q_i\equiv 1\pmod 4$ if $q_i$ is odd, which has to be the case for at least $k-1$ of these primes. Thus $4^{k-1}\mid u_m$. Again, we divide into two cases: $P$ is odd and $P$ is even.
\begin{case}{$P$ is odd}
\newline
\newline
Since $2^{2k-2}\mid u_m$, we can see that $2^{2k-4}\mid m$ so that $2^{2k-4}\leq m<n$. Thus $2^{2k}<16n$. So $3^k<2^{1.6k}<(16n)^{0.8}$. Also, $2k-4<\frac{\log n}{\log 2}$ so $k<\frac{\log(16n)}{2\log 2}$. Substituting into \eqref{eqn8}, we obtain
\begin{align*}
(16n)^{0.8}&>(n-2)\left(\frac{2\log\alpha}{\log\left(\alpha^{1.85}\left(\displaystyle\frac{\log(16n)}{\log 2}\right)\log\log n\right)}\right)-1\\
&=(n-2)\left(\frac{1.85\log\alpha+\log\log(16n)-\log\log 2+\log\log\log n}{2\log\alpha}\right)^{-1}-1\\
&>(n-2)\left(0.925+0.16+\frac{\log\log(16n)+\log\log\log n}{2\log\alpha}\right)^{-1}-1\\
&\geq(n-2)\left(1.085+\frac{\log\log(16n)+\log\log\log n}{2.38}\right)^{-1}-1.
\end{align*}
Note that
\begin{align*}
\log\log(16n)&=\log(\log 16+\log n)\\
&=\log\left(\left(1+\frac{\log 16}{\log n}\right)\log n\right)\\
&\leq\log\left(\left(\frac{105}{104}\right)\log n\right)\\
&<0.01+\log\log n,
\end{align*}
so that
\begin{equation*}
(16n)^{0.8}>(n-2)\left(1.09+\frac{\log\log n}{1.19}\right)^{-1}-1.
\end{equation*}
This implies that $n<3.4\cdot 10^7<2^{416}$, a contradiction.
\end{case}
\begin{case}{$P$ is even}
\newline
\newline
Since $2^{2k-2}\mid u_m$, we can see that $2^{2k-t_1-1}\mid m$, so that $2^{2k-t_1-1}\leq m<n$. Since $n$ is odd, we have that $u_n$ is odd, so that $q_i\equiv 1\pmod 4$ for all $1\leq i\leq k$. Thus, $q_i\geq 5$ for all $1\leq i\leq k$. We therefore deduce
\begin{equation*}
2^{t_1}\leq P<\alpha<\frac{u_n}{u_m}=\prod_{i=1}^k\left(\frac{q_i}{q_i-1}\right)\leq\left(\frac{5}{4}\right)^k.
\end{equation*}
Thus,
\begin{equation*}
2k-t_1-1\geq 2k-\frac{k\log\left(\frac{5}{4}\right)}{\log 2}-1>1.67k-1.
\end{equation*}
Therefore, $2^{1.67k-1}<n$, or $k<\frac{\log(2n)}{1.67\log 2}$. Then, also, $3^k<2^{1.59k}<(2n)^{1.59/1.67}<(2n)^{0.96}$. Substituting into \eqref{eqn8}, we obtain
\begin{align*}
(2n)^{0.96}&>(n-2)\left(\frac{2\log\alpha}{\log\left(\displaystyle\frac{2\log(2n)\alpha^{1.85}\log\log n}{1.67\log 2}\right)}\right)-1\\
&=(n-2)\left(\frac{\log 2+\log\log(2n)+1.85\log\alpha+\log\log\log n-\log 1.67-\log\log 2}{2\log\alpha}\right)^{-1}-1\\
&>(n-2)\left(\frac{0.55+\log\log(2n)+\log\log\log n}{2\log\alpha}+0.925\right)^{-1}-1\\
&>(n-2)\left(1.157+\frac{\log\log(2n)}{1.19}\right)^{-1}-1.
\end{align*}
Note that
\begin{align*}
\log\log(2n)&=\log(\log 2+\log n)\\
&=\log\left(\left(1+\frac{\log 2}{\log n}\right)\log n\right)\\
&<\log\left(\left(\frac{417}{416}\right)\log n\right)\\
&<0.003+\log\log n,
\end{align*}
so that
\begin{equation*}
(2n)^{0.96}>(n-2)\left(1.16+\frac{\log\log n}{1.19}\right)^{-1}-1.
\end{equation*}
This implies that $n<4\cdot 10^{23}<2^{416}$, a contradiction.
\end{case}
\section{Bounding $l$ Further}
Here we substantially improve our bound on $l$. Since we have eliminated the case of $n$ being odd, we assume $n$ is even throughout the rest of the paper. Write $n=2^sr_1^{\lambda_1}\cdots r_t^{\lambda_t}=:2^sn_1$ as the prime factorisation of $n$ with $s\geq 1$ and $3\leq r_1<r_2<\ldots<r_t$. From \eqref{eqn1}, we obtain the following:
\begin{align*}
l\log\alpha&<\sum_{p\mid u_n}\log\left(1+\frac{1}{p-1}\right)\\
&<\log\left(\frac{15}{4}\right)+\sum_{\substack{p\mid u_n\\p\geq 7}}\log\left(1+\frac{1}{p-1}\right)\\
&<\log\left(\frac{15}{4}\right)+\sum_{\substack{p\mid u_n\\p\geq 7}}\frac{1}{p-1}\\
&\leq\log\left(\frac{15}{4}\right)+\sum_{7\leq p\leq\alpha^4+1}\frac{1}{p-1}+\sum_{\substack{p\mid u_n\\p>\alpha^4+1}}\frac{1}{p-1}\\
&<\log\left(\frac{15}{4}\right)+\sum_{7\leq p\leq\alpha^4+1}\frac{1}{p}+\sum_{\substack{p\mid u_n\\p>\alpha^4+1}}\frac{1}{p}+\sum_{p\geq 7}\frac{1}{p(p-1)}\\
&<\log\left(\frac{15}{4}\right)+\sum_{7\leq p\leq\alpha^4+1}\frac{1}{p}+\sum_{7\leq p\leq 547}\frac{1}{p(p-1)}+\frac{1}{546}+\sum_{\substack{p\mid u_n\\p>\alpha^4+1}}\frac{1}{p}\\
&<1.38+\sum_{7\leq p\leq\alpha^4+1}\frac{1}{p}+\sum_{\substack{p\mid u_n\\p>\alpha^4+1}}\frac{1}{p}.
\end{align*}
By (3.18) in \cite{rosser}, we have
\begin{equation*}
\sum_{p\leq x}\frac{1}{p}<\log\log x+0.2615+\frac{1}{2\log^2(x)}<\log\log x+0.2772.
\end{equation*}
for $x\geq 286$ and, by computation, we can confirm that
\begin{equation*}
\sum_{p\leq x}\frac{1}{p}<\log\log x+0.2965
\end{equation*}
for $80\leq x\leq 286$. Therefore,
\begin{equation*}
\sum_{7\leq p\leq\alpha^4+1}\frac{1}{p}<\frac{1}{\alpha^4}+\sum_{7\leq p\leq\alpha^4}\frac{1}{p}<\frac{1}{119}+\log\log\alpha+\log 4+0.2965-\frac{1}{2}-\frac{1}{3}-\frac{1}{5}<\log\log\alpha+0.658.
\end{equation*}
Thus,
\begin{equation*}
l\log\alpha<\log\log\alpha+1.38+0.658+\sum_{\substack{p\mid u_n\\p>\alpha^4+1}}\frac{1}{p}=\log\log\alpha+2.038+\sum_{\substack{p\mid u_n\\p>\alpha^4+1}}\frac{1}{p}.
\end{equation*}
Notice that for $d\mid n$, we have
\begin{equation*}
\alpha^{4\left|\mathcal{P}_d'\right|}<\prod_{p\in\mathcal{P}_d'}p\leq u_d<\alpha^{d-1}
\end{equation*}
so that $\left|\mathcal{P}_d'\right|<\frac{d-1}{4}$. We therefore derive
\begin{equation*}
\sum_{\substack{d\mid n\\d\leq\alpha^2}}T_d<\sum_{d=2}^{\lfloor\alpha^2\rfloor}\frac{\left|\mathcal{P}_d'\right|}{\alpha^4}<\sum_{d=2}^{\lfloor\alpha^2\rfloor}\frac{d-1}{4\alpha^4}<\frac{1}{8},
\end{equation*}
so that
\begin{equation}\label{eqn12}
l\log\alpha<\log\log\alpha+2.163+\sum_{\substack{d\mid n\\d>\alpha^2}}T_d.
\end{equation}
We first show that $t\geq 1$, i.e. $n$ is not a power of $2$. Suppose to the contrary that $n=2^s$. Since $\alpha^2>10$, we therefore have
\begin{equation*}
l\log\alpha-\log\log\alpha<2.163+\sum_{i=4}^{\infty}\frac{2\log\left(2^i\right)}{2^i}=2.163+3\log 2-\frac{\log 4}{2}-\frac{\log 8}{4}<3.03.
\end{equation*}
Since $l$ is even and $\log\alpha>1.19$, this gives $l=2$. Suppose $P$ is odd. Then $2^{s+2}\mid u_n$, so $2^{s+1}\mid u_m$. Then $2^{s-1}\mid m$ so that $2^{s-1}\mid l$. So $s=1$ or $2$, so $n\leq 4$, contradicting $n>2^{416}$. Suppose $P$ is even. Then $2^{s+t_1-1}\mid u_n$, so $2^{s+t_1-2}\mid u_m$. Then $2^{s-1}\mid m$ so that $2^{s-1}\mid l$, again, contradicting $n>2^{416}$.
\newline
\newline
Thus, $t\geq 1$. Let
\begin{equation*}
\mathcal{I}:=\{i: r_i\mid m\}\text{ and }\mathcal{J}:=\{i: r_i\nmid m\}.
\end{equation*}
and
\begin{equation*}
M:=\prod_{i\in\mathcal{I}}r_i.
\end{equation*}
We can see that
\begin{equation*}
\sum_{\substack{d\mid n\\d>\alpha^2}}T_d=L_1+L_2,
\end{equation*}
where
\begin{equation*}
L_1:=\sum_{\substack{d\mid n\\r\mid d\implies r\mid 2M\\d>\alpha^2}}T_d\text{ and }L_2:=\sum_{\substack{d\mid n\\r_i\mid d\text{ for some }i\in\mathcal{J}\\d>\alpha^2}}T_d
\end{equation*}
where $r$ is prime. Let
\begin{equation*}
n':=2^s\prod_{i\in\mathcal{I}}r_i^{\lambda_i}
\end{equation*}
and notice that all divisors of $n'$ occur as divisors in the sum $L_1$. By Lemmas \ref{lem6} and \ref{lem8}, we therefore have
\begin{align*}
L_1&\leq\sum_{\substack{d\mid n'\\d>\alpha^2}}\frac{2\log d}{d}\\
&\leq\sum_{\substack{d\mid 4n'\\d>\alpha^2}}\frac{2\log d}{d}\\
&\leq\left(\sum_{d\mid 4n'}\frac{2\log d}{d}\right)-\log 2-\frac{\log 4}{2}-\frac{\log 8}{4}\\
&<2\left(\sum_{r\mid 2M}\frac{\log r}{r-1}\right)\left(\frac{2M}{\varphi(2M)}\right)-\log 2-\frac{\log 4}{2}-\frac{\log 8}{4}.
\end{align*}
Let $r_j$ be the smallest prime in $\mathcal{J}$ if $\mathcal{J}\neq\phi$. Since $s\geq 1$, we may write
\begin{equation*}
u_n=u_{n_1}v_{n_1}v_{2n_1}\cdots v_{2^{s-1}n_1},
\end{equation*}
by Lemma \ref{lem1}. If $q$ is any odd prime factor of $v_{n_1}$, then taking the equality in Lemma \ref{lem1} $\pmod q$ gives $Du_{n_1}^2\equiv 4\pmod q$. Therefore, $D$ is a quadratic residue $\pmod q$, so that $z(q)\mid q-1$. By Lemmas \ref{lem1} and \ref{lem15}, $v_d$ has a primitive divisor for every odd $d>3$. Noticing that $v_3=P\left(P^2+3\right)$, we can see that $v_3$ must have a prime factor that doesn't divide $P$, so, for all odd $d>1$, $v_d$ has a primitive prime factor, say $q_d$, with $d\mid q_d-1$. Observing that the number of divisors $d$ of $n_1$ that are multiples of $r_j$ is $\tau\left(\frac{n_1}{r_j}\right)$ gives
\begin{equation}\label{eqn40}
e_{r_j}=\nu_{r_j}\left(\varphi\left(u_n\right)\right)\geq\nu_{r_j}\left(\varphi\left(u_{n_1}\right)\right)\geq\tau\left(\frac{n_1}{r_j}\right)\geq\frac{\tau\left(n_1\right)}{2},
\end{equation}
with the last inequality following from
\begin{equation*}
\frac{\tau\left(\displaystyle\frac{n_1}{r_j}\right)}{\tau\left(n_1\right)}=\frac{\lambda_j}{\lambda_j+1}\geq\frac{1}{2}.
\end{equation*}
If $r_j\mid P^2+4$, then we have that $z\left(r_j\right)=r_j$. But then since $r_j\nmid m$, we have $\nu_{r_j}\left(\varphi\left(u_n\right)\right)=0$, so that $\tau\left(n_1\right)=0$, which is impossible. Therefore, $r_j\nmid P^2+4$, and we may use Lemma \ref{lem4} to conclude that
\begin{equation}\label{eqn9}
\tau\left(n_1\right)\leq\frac{\left(r_j+1\right)\log\alpha+0.2}{\log r_j}.
\end{equation}
Then we also have
\begin{align*}
L_2&\leq 2\tau\left(n_1\right)\sum_{0\leq a\leq s}\frac{\log\left(2^ar_j\right)}{2^ar_j}\\
&<\tau\left(n_1\right)\frac{\left(4\log 2+4\log r_j\right)}{r_j}\\
&\leq\frac{\left(r_j+1\right)\log\alpha+0.2}{\log r_j}\cdot\frac{\left(4\log 2+4\log r_j\right)}{r_j}\\
&<\log\alpha\left(1+\frac{1.169}{r_j}\right)\left(4+\frac{4\log 2}{\log r_j}\right).
\end{align*}
Then
\begin{align}
l\log\alpha&<\log\log\alpha+2.163+2\left(\sum_{r\mid 2M}\frac{\log r}{r-1}\right)\left(\frac{2M}{\varphi(2M)}\right)-\log 2-\frac{\log 4}{2}-\frac{\log 8}{4}\nonumber\\
&\quad+\log\alpha\left(1+\frac{1.169}{r_j}\right)\left(4+\frac{4\log 2}{\log r_j}\right)\label{eqn10}
\end{align}
if $\mathcal{J}\neq\phi$ and
\begin{equation}\label{eqn13}
l\log\alpha<\log\log\alpha+2.163+2\left(\sum_{r\mid 2M}\frac{\log r}{r-1}\right)\left(\frac{2M}{\varphi(2M)}\right)-\log 2-\frac{\log 4}{2}-\frac{\log 8}{4}
\end{equation}
otherwise. Assume that $\mathcal{J}\neq\phi$ for the moment. We have
\begin{equation*}
\frac{L_2}{\log\alpha}<\left(1+\frac{1.169}{r_j}\right)\left(4+\frac{4\log 2}{\log r_j}\right):=g\left(r_j\right).
\end{equation*}
Since $g(x)$ is decreasing  for $x\geq 3$, we deduce that
\begin{equation*}
L_2\leq g(3)\log\alpha<9.066\log\alpha,
\end{equation*}
so that
\begin{align}
l\log\alpha&<\log\log\alpha+2.163-\log 2-\frac{\log 4}{2}-\frac{\log 8}{4}+2\left(\sum_{r\mid 2M}\frac{\log r}{r-1}\right)\left(\frac{2M}{\varphi(2M)}\right)+9.066\log\alpha\nonumber\\
&<\log\log\alpha+0.2569+2\left(\sum_{r\mid 2M}\frac{\log r}{r-1}\right)\left(\frac{2M}{\varphi(2M)}\right)+9.066\log\alpha\label{eqn14},
\end{align}
by \eqref{eqn10}. Notice that \eqref{eqn14} also holds if $\mathcal{J}=\phi$. For any $N\in\mathbb{N}$, let
\begin{equation*}
f(N):=N\log\alpha-\log\log\alpha-0.2569-2\left(\sum_{r\mid N}\frac{\log r}{r-1}\right)\left(\frac{N}{\varphi(N)}\right).
\end{equation*}
Then we have
\begin{equation}\label{eqn11}
f(l)<9.066\log\alpha
\end{equation}
by \eqref{eqn14}. Assuming $l\geq 26$, we therefore have
\begin{equation*}
l\log\alpha-\log\log\alpha-0.2569-2(\log 2)\left(1.79\log\log l+\frac{2.5}{\log\log l}\right)\left(\frac{\log l}{\log\log l-1.1714}\right)<9.066\log\alpha,
\end{equation*}
by Lemma \ref{lem5}. Thus, assuming $l\geq 26$ and dividing through by $\log\alpha$ gives
\begin{equation*}
l-\frac{\log\log\alpha}{\log\alpha}-\frac{0.2569}{\log\alpha}-\frac{2\log 2}{\log\alpha}\left(1.79\log\log l+\frac{2.5}{\log\log l}\right)\left(\frac{\log l}{\log\log l-1.1714}\right)<9.066,
\end{equation*}
so that
\begin{equation*}
l-\frac{0.2569}{\log\alpha}-\frac{2\log 2}{\log\alpha}\left(1.79\log\log l+\frac{2.5}{\log\log l}\right)\left(\frac{\log l}{\log\log l-1.1714}\right)<9.466,
\end{equation*}
Since $\log\alpha>1.19$, we may substitute in $\log\alpha=1.19$ into the above inequality for it to remain true. Once we do that, we get an inequality that implies $l<82$, so that $l\leq 80$. As well, we may also substitute in $\log\alpha=1.19$ into \eqref{eqn11} for it to remain true. Once we do that and test out all even values of $l\leq 80$, we discover only even values of $l\leq 12$ work. So $l\leq 12$. We now deduce two lemmas.
\begin{lemma}\label{lem11}
We have $2^s\mid l$. Therefore, $s\leq 3$.
\end{lemma}
\begin{proof}
We may assume $s\geq 2$. Let $q_1,q_2,\ldots,q_{2s-2}$ be primitive prime factors of $u_4,u_8,u_{16},\ldots,u_{2^s},u_{4n_1}$, $u_{8n_1},\ldots,u_{2^sn_1}$ respectively. Then $q_1,q_2,\ldots,q_{2s-2}$ must all be odd. We have
\begin{equation*}
Pq_1q_2\cdots q_{2s-2}u_{n_1}\mid u_n,
\end{equation*}
and so 
\begin{equation*}
\varphi(P)\left(q_1-1\right)\left(q_2-1\right)\cdots\left(q_{2s-2}-1\right)\varphi\left(u_{n_1}\right)\mid u_m.
\end{equation*}
So we have $2^{2s-2}\varphi(P)\varphi\left(u_{n_1}\right)\mid u_m$. Suppose $P$ is odd. Then we have $2^{2s}\mid u_m$ since $P\geq 3$ and $u_{n_1}\geq u_3>2$. Thus, by Lemma \ref{lem9}, we have $2^{2s-2}\mid m$. Since $s\geq 2$, this implies that $2^s\mid m$, so that $2^s\mid\gcd(m,n)$. Thus $2^s\mid l$. Now suppose $P$ is even. Then $2^{2s-2+t_1}\mid u_m$, so that $2^{2s-1}\mid m$, by Lemma \ref{lem9}. Then $2^s\mid m$, so that again $2^s\mid l$. From our possible values of $l$, we can therefore see that $s\leq 3$.
\end{proof}
\begin{lemma}\label{lem12}
If $r$ is an odd prime with $r^{\lambda}\mid n$ where $\lambda>e_r$, then $r^{\lambda-e_r}\mid l$.
\end{lemma}
\begin{proof}
Since $n$ is even, $2r^{\lambda}\mid n$. Let $q$ be a primitive prime factor of $u_{2r^{\lambda}}$. Then $q\mid v_{r^{\lambda}}$. By Lemma \ref{lem1}, we can therefore deduce that $Du_{r_{\lambda}}^2\equiv 4\pmod q$, so that $D$ is a quadratic residue $\pmod{q}$. Therefore, by Lemma \ref{lem18}, $q\mid u_{q-1}$, so that $2r^{\lambda}\mid q-1$. Therefore, $2r^{\lambda}\mid u_m$. From Lemma \ref{lem2}, this implies that $r^{\lambda-e_p}\mid m$. Since $r^{\lambda-e_p}\mid n$, we have that $r^{\lambda-e_p}\mid l$.
\end{proof}
From  our possible values of $l$, we deduce that if $M>1$, then $M=\left\{r_i\right\}$ where $r_i=3$ or $5$, with $\lambda_i\leq e_{r_i}+1$. Altogether, we have proved the following so far.
\begin{lemma}\label{lem13}
If $\varphi\left(u_n\right)=u_m$ with $n>1$ and $n=2^sr_1^{\lambda_1}\cdots r_t^{\lambda_t}$, then $1\leq s\leq 3$, $n>m\geq 2^{416}$, and $\gcd(m,n)=2,4,6,8,10,12$. Also $l=n-m=2,4,6,8,10,12$. Finally, $\lambda_i\leq e_{r_i}$ for all $r_i\nmid m$ and $\lambda_i\leq e_{r_i}+1$ for all $r_i\mid m$.
\end{lemma}
\section{The Case of $P\geq 80$}
Here we completely eliminate the possibility of $P\geq 80$. First, we derive a lower bound on $L_2$.
\begin{lemma}\label{lem14}
For all $P\geq 3$, we have $L_2>\log\alpha-\log\log\alpha-2.163$ if $M=1$ and $L_2>6\log\alpha-\log\log\alpha-5.894$ if $M>1$.
\end{lemma}
\begin{proof}
First, by \eqref{eqn12}, we have
\begin{equation*}
L_2>l\log\alpha-\log\log\alpha-2.163-L_1.
\end{equation*}
By Lemma \ref{lem13}, we know that $M=1,3,5$. Consider the case of $M=1$. Then we have
\begin{equation*}
L_1=\sum_{\substack{d\mid n\\r\mid d\implies r\mid 2\\d>\alpha^2}}T_d
\end{equation*}
All divisors $d$ in the above sum are powers of $2$. However, the highest power of $2$ dividing $n$ is at most $8$ by Lemma \ref{lem11}, and $8<\alpha^2$. Therefore, $L_1=0$ and the result follows for this case. Consider $M=3$. Then
\begin{align*}
L_1&\leq\sum_{\substack{d\mid n'\\d>\alpha^2}}\frac{2\log d}{d}\\
&\leq\sum_{\substack{d\mid 12n'\\d>\alpha^2}}\frac{2\log d}{d}\\
&\leq\left(\sum_{d\mid 12n'}\frac{2\log d}{d}\right)-\log 2-\frac{2\log 3}{3}-\frac{\log 4}{2}-\frac{\log 6}{3}-\frac{\log 8}{4}-\frac{2\log 9}{9}\\
&<2\left(\sum_{r\mid 2M}\frac{\log r}{r-1}\right)\left(\frac{2M}{\varphi(2M)}\right)-\log 2-\frac{2\log 3}{3}-\frac{\log 4}{2}-\frac{\log 6}{3}-\frac{\log 8}{4}-\frac{2\log 9}{9}\\
&=6\left(\log 2+\frac{\log(3)}{2}\right)-\log 2-\frac{2\log 3}{3}-\frac{\log 4}{2}-\frac{\log 6}{3}-\frac{\log 8}{4}-\frac{2\log 9}{9}\\
&<3.731.
\end{align*}
Then
\begin{equation*}
L_2>6\log\alpha-\log\log\alpha-5.894
\end{equation*}
Likewise, if $M=5$, we have
\begin{align*}
L_1&\leq\sum_{\substack{d\mid n'\\d>\alpha^2}}\frac{2\log d}{d}\\
&\leq\sum_{\substack{d\mid 4n'\\d>\alpha^2}}\frac{2\log d}{d}\\
&\leq\left(\sum_{d\mid 4n'}\frac{2\log d}{d}\right)-\log 2-\frac{\log 4}{2}-\frac{2\log 5}{5}-\frac{\log 8}{4}-\frac{\log 10}{5}\\
&<2\left(\sum_{r\mid 2M}\frac{\log r}{r-1}\right)\left(\frac{2M}{\varphi(2M)}\right)-\log 2-\frac{\log 4}{2}-\frac{2\log 5}{5}-\frac{\log 8}{4}-\frac{\log 10}{5}\\
&=5\left(\log 2+\frac{\log 5}{4}\right)-\log 2-\frac{\log 4}{2}-\frac{2\log 5}{5}-\frac{\log 8}{4}-\frac{\log 10}{5}\\
&<2.468,
\end{align*}
so that
\begin{equation*}
L_2>10\log\alpha-\log\log\alpha-4.631>6\log\alpha-\log\log\alpha-5.894.
\end{equation*}
The result follows.
\end{proof}
We divide into two cases: $r_j\leq P$ or $r_j>P$.
\subsection{The Case of $r_j\leq P$}
By \eqref{eqn12}, we have
\begin{equation*}
2\log\alpha<\log\log\alpha+2.163+\sum_{\substack{d\mid n\\d>\alpha^2}}S_d.
\end{equation*}
For divisors $d$ involved in the above sum, we have $d>\alpha^2>80^2>40$, so that Lemma \ref{lem16} applies. We have
\begin{equation*}
S_d<\frac{3.204}{d}+\frac{1}{d\log d}+\frac{4\log\log d}{\varphi(d)}+\frac{4\log\log\alpha}{\varphi(d)\log d}.
\end{equation*}
Lemma \ref{lem5} gives
\begin{equation*}
S_d<\frac{3.204}{d}+\frac{1}{d\log d}+\frac{7.16(\log\log d)^2}{d}+\frac{7.16\log\log d}{d\log d}\log\log\alpha+\frac{10}{d}+\frac{10\log\log\alpha}{d(\log d)(\log\log d)}.
\end{equation*}
Since $d>\alpha^2$, we deduce
\begin{equation*}
S_d<\frac{3.204}{\alpha^2}+\frac{1}{2\alpha^2\log\alpha}+\frac{7.16\left(\log\log\left(\alpha^2\right)\right)^2}{\alpha^2}+\frac{7.16\log\log\left(\alpha^2\right)}{2\alpha^2\log\alpha}\log\log\alpha+\frac{10}{\alpha^2}+\frac{10\log\log\alpha}{2\alpha^2\log\alpha\left(\log\log\left(\alpha^2\right)\right)}.
\end{equation*}
Since $r_j\leq P$, by \eqref{eqn9}, we have
\begin{equation*}
\tau\left(n_1\right)\leq\frac{(\alpha+1)\log\alpha+0.2}{\log\alpha}
\end{equation*}
Since $s\leq 3$, this gives
\begin{equation*}
\tau(n)\leq 4\tau\left(n_1\right)\leq\frac{4(\alpha+1)\log\alpha+0.8}{\log\alpha}=4(\alpha+1)+\frac{0.8}{1.19}=4\alpha+4.673<5\alpha.
\end{equation*}
Therefore, since $\alpha>80$, we have
\begin{align*}
\sum_{\substack{d\mid n\\d>\alpha^2}}S_d&<\frac{16.02}{\alpha}+\frac{2.5}{\alpha\log\alpha}+\frac{35.8\left(\log\log\left(\alpha^2\right)\right)}{\alpha}+\frac{17.9\log\log\left(\alpha^2\right)}{\alpha\log\alpha}\log\log\alpha+\frac{50}{\alpha}+\frac{25\log\log\alpha}{\alpha\log\alpha\left(\log\log\left(\alpha^2\right)\right)}\\
&<2.956.
\end{align*}
Thus
\begin{equation*}
2\log\alpha<\log\log\alpha+2.163+2.956=\log\log\alpha+5.119.
\end{equation*}
But since $\alpha>80$, we also have
\begin{equation*}
2\log\alpha-\log\log\alpha>2\log 80-\log\log 80>7.286,
\end{equation*}
a contradiction.
\subsection{The Case of $r_j>P$}
For this case, we again divide into two cases: $M=1$ and $M>1$.
\setcounter{case}{0}
\begin{case}{$M=1$}
\newline
\newline
Since $M=1$, we have $r_j=r_1$. By Lemma \ref{lem14}, we have
\begin{equation}\label{eqn36}
2\log\alpha-\log\log\alpha-2.163<\sum_{\substack{\substack{d\mid n\\r_i\mid d\text{ for some }i\in\mathcal{J}}}}S_d. 
\end{equation}
For every divisor $d$ involved in the above sum, we have $d\geq r_1>P\geq 80$, so Lemma \ref{lem16} applies. Since $\alpha-P<1$, we have $r_1>\alpha$. Also, since $r_1$ is prime, we have $r_1\geq 83$. We have
\begin{align*}
S_d&<\frac{3.204}{d}+\frac{1}{d\log d}+\frac{4\log\log d}{\varphi(d)}+\frac{4\log\log\alpha}{\varphi(d)\log d}\\
&<\frac{3.204}{\varphi(d)}+\frac{1}{\varphi(d)\log 80}+\frac{4\log\log d}{\varphi(d)}+\frac{4\log\log\alpha}{\varphi(d)\log\alpha}\\
&<\frac{3.204}{\varphi(d)}+\frac{1}{\varphi(d)\log 80}+\frac{4\log\log d}{\varphi(d)}+\frac{4\log\log 80}{\varphi(d)\log 80}\\
&<\frac{4.781}{\varphi(d)}+\frac{4\log\log d}{\varphi(d)}\\
&<\frac{7.236\log\log d}{\varphi(d)}.
\end{align*}
Notice that every divisor $d$ in the sum in \eqref{eqn36} can be written as $2^ad_1$ where $0\leq a\leq 3$ and $d_1$ is a divisor of $n_1$. Therefore, we have
\begin{align*}
L_2&<7.236\sum_{\substack{d\mid n_1\\d>1}}\left(\frac{\log\log d}{\varphi(d)}+\frac{\log\log(2d)}{\varphi(2d)}+\frac{\log\log(4d)}{\varphi(4d)}+\frac{\log\log (8d)}{\varphi(8d)}\right)\\
&=7.236\sum_{\substack{d\mid n_1\\d>1}}\left(\frac{\log\log d}{\varphi(d)}+\frac{\log\log(2d)}{\varphi(d)}+\frac{\log\log(4d)}{2\varphi(d)}+\frac{\log\log (8d)}{4\varphi(d)}\right).
\end{align*}
For $1\leq a\leq 3$, we have
\begin{align*}
\log\log\left(2^ad\right)&=\log\log d+\log\left(1+\frac{a\log 2}{\log d}\right)\\
&<\log\log d+\frac{a\log 2}{\log d}\\
&\leq\log\log d\left(1+\frac{a\log 2}{(\log 83)(\log\log 83)}\right)
\end{align*}
We deduce
\begin{equation}\label{eqn23}
L_2<22\sum_{\substack{d\mid n_2\\d>1}}\frac{\log\log d}{\varphi(d)}.
\end{equation}
By \eqref{eqn9}, we have
\begin{equation}\label{eqn24}
\tau\left(n_1\right)<\frac{\left(r_1+1\right)\log r_1+0.2}{\log r_1}<2r_1,
\end{equation}
so that
\begin{equation*}
t<\frac{\log\left(\tau\left(n_1\right)\right)}{\log 2}<\frac{\log\left(2r_1\right)}{\log 2}:=g\left(r_1\right).
\end{equation*}
By Lemma $4.8$ in \cite{chen}, we have that the function $\log\log\left(x_1x_2\right)\leq\log\log\left(x_1\right)\log\log\left(x_2\right)$ for any pair of real numbers $x_1,x_2\geq 78$. Since $r_1\geq 83$, we therefore deduce the following:
\begin{align*}
\sum_{\substack{d\mid n_1\\d>1}}\frac{\log\log d}{\varphi(d)}&<\prod_{i=1}^t\left(1+\frac{\log\log\left(r_i\right)}{\varphi\left(r_i\right)}+\frac{\log\log\left(r_i^2\right)}{\varphi\left(r_i^2\right)}+\ldots\right)-1\\
&\leq\prod_{i=1}^t\left(1+\frac{\log\log\left(r_i\right)}{\varphi\left(r_i\right)}+\frac{\left(\log\log\left(r_i\right)\right)^2}{\varphi\left(r_i^2\right)}+\ldots\right)-1\\
&=\prod_{i=1}^t\left(1+\frac{\log\log r_i}{r_i-1}\cdot\frac{1}{1-\displaystyle\frac{\log\log r_i}{r_i}}\right)-1\\
&\leq\left(1+\frac{\log\log r_1}{r_1-1}\cdot\frac{1}{1-\displaystyle\frac{\log\log r_1}{r_1}}\right)^{g\left(r_1\right)}-1\\
&:=T-1.
\end{align*}
By $r_1\geq 83$,
\begin{align*}
\log T&=g\left(r_1\right)\log\left(1+\frac{\log\log r_1}{r_1-1}\cdot\frac{1}{1-\displaystyle\frac{\log\log r_1}{r_1}}\right)\\
&<g\left(r_1\right)\frac{\log\log r_1}{r_1-1}\cdot\frac{1}{1-\displaystyle\frac{\log\log r_1}{r_1}}\\
&=\frac{\log\left(2r_1\right)}{\log 2}\cdot\frac{\log\log r_1}{r_1-1}\cdot\frac{r_1}{r_1-\log\log r_1}\\
&=\frac{1}{\log 2}\cdot\frac{\log\left(2r_1\right)}{\sqrt{r_1}}\cdot\frac{\log\log r_1}{\sqrt{r_1}}\cdot\frac{r_1}{r_1-1}\cdot\frac{r_1}{r_1-\log\log r_1}\\
&\leq\frac{1}{\log 2}\cdot\frac{\log 166}{\sqrt{83}}\cdot\frac{\log\log 83}{\sqrt{83}}\cdot\frac{83}{82}\cdot\frac{83}{83-\log\log 83}\\
&<0.137.
\end{align*}
Thus,
\begin{equation*}
\sum_{\substack{d\mid n_1\\d>1}}\frac{\log\log d}{\varphi(d)}<e^{0.137}-1<0.147.
\end{equation*}
Therefore,
\begin{equation*}
L_2<22\cdot 0.147<3.234.
\end{equation*}
Hence,
\begin{equation*}
2\log\alpha-\log\log\alpha<5.397.
\end{equation*}
But since $\alpha>80$, we also have
\begin{equation*}
2\log\alpha-\log\log\alpha>2\log 80-\log\log 80>7.286,
\end{equation*}
a contradiction.
\end{case}
\begin{case}{$M>1$}
\newline
\newline
By Lemma \ref{lem14}, we have
\begin{equation*}
6\log\alpha-\log\log\alpha-5.894<\sum_{\substack{\substack{d\mid n\\r_i\mid d\text{ for some }i\in\mathcal{J}}}}S_d. 
\end{equation*}
For every divisor $d$ involved in the above sum, we can again derive that
\begin{equation*}
S_d<\frac{7.236\log\log d}{\varphi(d)}.
\end{equation*}
Wlog let $M=r_1$, $r_j=r_2$, and $n_2:=\frac{n_1}{r_1^{\lambda_1}}$. Again, we can similarly derive that
\begin{align*}
\sum_{\substack{d\mid n\\r_i\mid d\text{ for some }i\in\mathcal{J}}}S_d&<22\sum_{\substack{d\mid n_2\\d>1}}\sum_{y=1}^{\infty}\frac{\log\log\left(M^yd\right)}{\varphi\left(M^yd\right)}\\
&=22\sum_{\substack{d\mid n_2\\d>1}}\sum_{y=0}^{\infty}\frac{\log\log\left(M^yd\right)}{(M-1)M^{y-1}\varphi(d)}.
\end{align*}
For all $x\geq 83$, we can see that
\begin{equation*}
\frac{\log\log(3x)}{\log\log x}=1+\frac{\log\left(1+\frac{\log 3}{\log x}\right)}{\log\log x}\leq 1+\frac{\log\left(1+\frac{\log 3}{\log 83}\right)}{\log\log 83}<1.15.
\end{equation*}
Therefore,
\begin{align}
\sum_{\substack{d\mid n\\r_i\mid d\text{ for some }i\in\mathcal{J}}}S_d&<22\sum_{\substack{d\mid n_2\\d>1}}\sum_{y=0}^{\infty}\frac{1.15^y\log\log(d)}{2\cdot 3^{y-1}\varphi(d)}\nonumber\\
&<33\sum_{\substack{d\mid n_2\\d>1}}\frac{\log\log(d)}{\varphi(d)}\sum_{y=0}^{\infty}\left(\frac{1.15}{3}\right)^y\nonumber\\
&<53.52\sum_{\substack{d\mid n_2\\d>1}}\frac{\log\log(d)}{\varphi(d)}.\label{eqn38}
\end{align}
As before, we can derive that
\begin{equation*}
\sum_{\substack{d\mid n_2\\d>1}}\frac{\log\log(d)}{\varphi(d)}<0.147.
\end{equation*}
Therefore,
\begin{equation*}
6\log\alpha-\log\log\alpha-5.894<53.52\cdot 0.147,
\end{equation*}
so that
\begin{equation*}
6\log\alpha-\log\log\alpha<13.77.
\end{equation*}
But since $\alpha>80$, we also have
\begin{equation*}
6\log\alpha-\log\log\alpha>6\cdot\log 80-\log\log 80>24.81,
\end{equation*}
a contradiction.
\end{case}
\section{The Case of $3\leq P\leq 79$}
We now deal with the case of $3\leq P\leq 79$.
\begin{lemma}
If $3\leq P\leq 79$ and $d\mid n$ with $d\geq 173$, then
\begin{equation*}
\log\prod_{z(p)=d}\left(1+\frac{1}{p-1}\right)<\frac{7.39\log\log d}{\varphi(d)}
\end{equation*}
\end{lemma}
\begin{proof}
We have the following:
\begin{align*}
\log\prod_{z(p)=d}\left(1+\frac{1}{p-1}\right)&=\sum_{l_p=d}\log\left(1+\frac{1}{p-1}\right)\\
&<\sum_{z(p)=d}\frac{1}{p-1}\\
&=\sum_{z(p)=d}\frac{1}{p(p-1)}+\sum_{z(p)=d}\frac{1}{p}\\
&<\sum_{n=d-1}^{\infty}\frac{1}{n(n-1)}+\sum_{z(p)=d}\frac{1}{p}\\
&=\frac{1}{d-2}+S_d\\
&<\frac{1.012}{d}+S_d.
\end{align*}
By Lemma \ref{lem16}, we therefore have
\begin{align*}
\log\prod_{z(p)=d}\left(1+\frac{1}{p-1}\right)&<\frac{4.216}{d}+\frac{1}{d\log d}+\frac{4\log\log d}{\varphi(d)}+\frac{4\log\log\alpha}{\varphi(d)\log d}\\
&\leq\frac{4.216}{\varphi(d)}+\frac{1}{\varphi(d)\log 173}+\frac{4\log\log d}{\varphi(d)}+\frac{4\log\log 80}{\varphi(d)\log 173}\\
&<\frac{5.557}{\varphi(d)}+\frac{4\log\log d}{\varphi(d)}\\
&<\frac{7.39\log\log d}{\varphi(d)}.
\end{align*}
\end{proof}
\begin{lemma}
We have $r_j<173$.
\end{lemma}
\begin{proof}
Suppose for a contradiction that $r_j\geq 173$. We divide into cases: $M=1$ and $M>1$.
\setcounter{case}{0}
\begin{case}{$M=1$}
\newline
\newline
Again, in this case, we have $r_j=r_1$. By Lemma \ref{lem11}, we have
\begin{equation*}
\alpha^{2^s}\leq\alpha^l<\prod_{l_p\mid n}\left(1+\frac{1}{p-1}\right)
\end{equation*}
Since $M=1$ and the only factors of $n$ less than $173$ are the factors of $2^s$, we have
\begin{equation*}
\alpha^{2^s}<\frac{u_{2^s}}{\varphi\left(u_{2^s}\right)}\prod_{\substack{l_p\mid n\\l_p\geq 173}}\left(1+\frac{1}{p-1}\right).
\end{equation*}
So
\begin{align*}
\log\left(\frac{\alpha^{2^s}\varphi\left(u_{2^s}\right)}{u_{2^s}}\right)&<\sum_{\substack{l_p\mid n\\l_p\geq 173}}\log\left(1+\frac{1}{p-1}\right)\\
&<\sum_{\substack{d\mid n\\d\geq 173}}\frac{7.39\log\log d}{\varphi(d)}.
\end{align*}
Similarly to how we derived \eqref{eqn23}, we can derive that
\begin{equation*}
\sum_{\substack{d\mid n\\d\geq 173}}\frac{7.39\log\log d}{\varphi(d)}<21.99\sum_{d\mid n_1}\frac{\log\log d}{\varphi(d)}
\end{equation*}
As well, we again have
\begin{align*}
\sum_{\substack{d\mid n_1\\d>1}}\frac{\log\log d}{\varphi(d)}&<\left(1+\frac{\log\log r_1}{r_1-1}\cdot\frac{1}{1-\frac{\log\log r_1}{r_1}}\right)^{g\left(r_1\right)}-1\\
&:=T-1.
\end{align*}
By $r_1\geq 173$,
\begin{align*}
\log T&=g\left(r_1\right)\log\left(1+\frac{\log\log r_1}{r_1-1}\cdot\frac{1}{1-\displaystyle\frac{\log\log r_1}{r_1}}\right)\\
&<g\left(r_1\right)\frac{\log\log r_1}{r_1-1}\cdot\frac{1}{1-\displaystyle\frac{\log\log r_1}{r_1}}\\
&=\frac{\log\left(2r_j\right)}{\log 2}\cdot\frac{\log\log r_1}{r_1-1}\cdot\frac{r_1}{r_1-\log\log r_1}\\
&=\frac{1}{\log 2}\cdot\frac{\log\left(2r_j\right)}{\sqrt{r_1}}\cdot\frac{\log\log r_1}{\sqrt{r_1}}\cdot\frac{r_1}{r_1-1}\cdot\frac{r_1}{r_1-\log\log r_1}\\
&\leq\frac{1}{\log 2}\cdot\frac{\log 346}{\sqrt{173}}\cdot\frac{\log\log 173}{\sqrt{173}}\cdot\frac{173}{172}\cdot\frac{173}{173-\log\log 173}\\
&<0.082.
\end{align*}
Thus,
\begin{equation*}
\sum_{\substack{d\mid n_1\\d>1}}\frac{\log\log d}{\varphi(d)}<e^{0.082}-1<0.086.
\end{equation*}
So
\begin{equation}\label{eqn37}
\log\left(\frac{\alpha^{2^s}\varphi\left(u_{2^s}\right)}{u_{2^s}}\right)<21.99\cdot 0.086<1.892.
\end{equation}
It is easy to check that for $1\leq s\leq 3$ and $3\leq P\leq 5$, \eqref{eqn37} doesn't hold, so we may assume that $6\leq P\leq 79$. From \eqref{eqn37}, we have
\begin{equation*}
\frac{\alpha^{2^s}\varphi\left(u_{2^s}\right)}{u_{2^s}}<6.633.
\end{equation*}
From Lemma \ref{lem5}, we have
\begin{equation*}
\frac{\alpha^{2^s}}{\left(1.79\cdot\log\left(\log\left(u_{2^s}\right)\right)+\displaystyle\frac{2.5}{\log\log\left(u_{2^s}\right)}\right)}<6.633.
\end{equation*}
Since $u_{2^s}<\alpha^{2^s-1}$, this implies
\begin{equation*}
\frac{6\cdot u_{2^s}}{\left(1.79\cdot\log\left(\log\left(u_{2^s}\right)\right)+\displaystyle\frac{2.5}{\log\log\left(u_{2^s}\right)}\right)}<6.633.
\end{equation*}
But since $u_{2^s}>6$ we also have
\begin{align*}
\frac{6\cdot u_{2^s}}{\left(1.79\cdot\log\left(\log\left(u_{2^s}\right)\right)+\displaystyle\frac{2.5}{\log\log\left(u_{2^s}\right)}\right)}>\frac{36}{\left(1.79\cdot\log\log 6+\displaystyle\frac{2.5}{\log\log 6}\right)}>6.753,
\end{align*}
a contradiction.
\end{case}
\begin{case}{$M>1$}
\newline
\newline
Again, wlog let $M=r_1$ and $r_j=r_2$. Since the only factors less than $173$ are contained in the set of factors of $2^sr_1^{\lambda_1}$, we have
\begin{equation*}
\alpha^{2^sr_1}<\frac{u_{2^sr_1^{\lambda_1}}}{\varphi\left(u_{2^sr_1^{\lambda_1}}\right)}\prod_{\substack{l_p\mid n\\l_p\geq 173}}\left(1+\frac{1}{p-1}\right).
\end{equation*}
So
\begin{align*}
\log\left(\frac{\alpha^{2^sr_1^{\lambda_1}}\varphi\left(u_{2^s}r_1^{\lambda_1}\right)}{u_{2^sr_1^{\lambda_1}}}\right)&<\sum_{\substack{l_p\mid n\\l_p\geq 173}}\log\left(1+\frac{1}{p-1}\right)\\
&<\sum_{\substack{d\mid n\\d\geq 173}}\frac{7.39\log\log d}{\varphi(d)}
\end{align*}
Similarly to how we derived \eqref{eqn38}, we can derive that
\begin{equation*}
\sum_{\substack{d\mid n\\d\geq 173}}\frac{7.39\log\log d}{\varphi(d)}<53.49\sum_{d\mid n_2}\frac{\log\log d}{\varphi(d)}
\end{equation*}
where, again, $n_2:=n/n'$. As before, we can derive that
\begin{equation*}
\sum_{\substack{d\mid n_2\\d>1}}\frac{\log\log d}{\varphi(d)}<0.086,
\end{equation*}
so that
\begin{equation*}
\log\left(\frac{\alpha^{2^sr_1}\varphi\left(u_{2^sr_1^{\lambda_1}}\right)}{u_{2^sr_1^{\lambda_1}}}\right)<53.49\cdot 0.086<4.601.
\end{equation*}
By Maple, we can check that that $e_3,e_5\leq 5$ for all $3\leq P\leq 79$. So, by Lemma \ref{lem12}, we have ${r_1}^7\nmid n$. Therefore,
\begin{equation}\label{eqn39}
\log\left(\frac{\alpha^{2^sr_1}\varphi\left(u_{2^sr_1^6}\right)}{u_{2^sr_1^6}}\right)<4.601.
\end{equation}
From \eqref{eqn39}, we have
\begin{equation*}
\frac{\alpha^6\varphi\left(u_{8\cdot 5^6}\right)}{u_{8\cdot  5^6}}<\frac{\alpha^{2r_1}\varphi\left(u_{8\cdot r_1^6}\right)}{u_{8\cdot  r_1^6}}<99.59,
\end{equation*}
which implies
\begin{equation*}
\frac{\alpha^6}{1.79\cdot\log\log\left(u_{8\cdot 5^6}\right)+\displaystyle\frac{2.5}{\log\log\left(u_{8\cdot 5^6}\right)}}<99.59.
\end{equation*}
Therefore,
\begin{equation*}
\frac{\alpha^6}{1.79\log\log\left(\alpha^{8\cdot 5^6-1}\right)+\displaystyle\frac{2.5}{\log\log\left(\alpha^{8\cdot 5^6-1}\right)}}<99.59.
\end{equation*}
If $P\geq 4$, then we have $4<\alpha<80$, so that
\begin{equation*}
\frac{\alpha^6}{1.79\log\log\left(\alpha^{8\cdot 5^6-1}\right)+\displaystyle\frac{2.5}{\log\log\left(\alpha^{8\cdot 5^6-1}\right)}}>\frac{4^6}{1.79\log\log\left(80^{8\cdot 5^6-1}\right)+\displaystyle\frac{2.5}{\log\log\left(80^{8\cdot 5^6-1}\right)}}>171.8,
\end{equation*}
a contradiction. So assume $P=3$. Then $3.3<\alpha<3.303$. Since $r_1=3$ or $5$, we can see that $e_{r_1}=1$. Therefore, by Lemma \ref{lem12}, we have $\lambda_1\leq 2$. Therefore, we can similarly see that
\begin{align*}
99.59&>\frac{\alpha^6}{1.79\log\log\left(\alpha^{8\cdot 5^2-1}\right)+\displaystyle\frac{2.5}{\log\log\left(\alpha^{8\cdot 5^2-1}\right)}}\\
&>\frac{3.3^6}{1.79\log\log\left(3.303^{8\cdot 5^2-1}\right)+\displaystyle\frac{2.5}{\log\log\left(3.303^{8\cdot 5^2-1}\right)}}\\
&>125.9,
\end{align*}
a contradiction.
\end{case}
\end{proof}
Given $3\leq P\leq 79$ and $3\leq r_j\leq 167$ (since $167$ is the highest prime less than $173$), we can check with Maple that $e_{r_j}\leq 5$. So, by \eqref{eqn40}, we get
\begin{equation*}
\tau\left(n_1\right)\leq 2e_{r_j}\leq 10,
\end{equation*}
which implies that $t\leq 3$. Also, by \eqref{eqn40}, we get
\begin{equation*}
\tau\left(\frac{n_1}{r_j}\right)\leq e_{r_j}.
\end{equation*}
Hence, if $e_{r_j}=1$, then we have $n_1=r_j$, so that $n\leq 8\cdot n_1=8\cdot r_j\leq 8\cdot 167=1336$, contradicting $n>2^{416}$. Therefore, $e_{r_j}\geq 2$.
\begin{lemma}\label{lem17}
$n_1$ has a prime factor larger than $3\cdot 10^{30}$.
\end{lemma}
\begin{proof}
We know that $n>2^{416}$. Since $n\leq 8n_1$, we have $n_1>2^{413}$. Since $\tau\left(n_1\right)\leq 10$, we have $\prod_{i=1}^t\left(\lambda_i+1\right)\leq 10$. Since $\lambda_i\geq 1$ for all $1\leq i\leq t$, we can see that this must imply that $\sum_{i=1}^t\lambda_i\leq 5$. Therefore, $2^{413}<167r_t^4$ since $r_t$ is the largest prime factor of $n_1$ and the smallest prime factor of $n_1$ is at most $167$. This implies that $r_t>3\cdot 10^{30}$.
\end{proof}
\begin{lemma}
We have $t=2$.
\end{lemma}
\begin{proof}
Since the smallest prime factor of $n_1$ is at most $167$ and the largest prime factor of $n_1$ is larger than $3\cdot 10^{30}$, we have $t\geq 2$. We know that $t\leq 3$, so it remains to eliminate the possibility of $t=3$. Assume $t=3$ for a contradiction. Then $\tau\left(n_1\right)=8$ since $\tau\left(n_1\right)\leq 10$, so that $2e_{r_j}\geq 8$. Hence, $e_{r_j}\geq 4$. Checking with Maple, we see that this implies $r_j=3$, and $P=22$ or $P=59$. We divide into cases.
\setcounter{case}{0}
\begin{case}{$P=22$}
\newline
\newline
For $P=22$, we use Maple to see that $e_3=5$, $l_3=4$, and that $487\mid 10714=v_3$, where $487$ is prime. Therefore, $487\mid u_n$. Since $3^5\mid 486$, this gives $3^5\mid u_m$. Let $r'$ be a prime factor of $\frac{n_1}{3}$. Then $3r'\mid n_1$, and we have $v_{3r'}\mid u_n$. Also, $v_{3r'}$ has a primitive prime factor, say $q$. Then $q\equiv 1\pmod{3r'}$, so $3\mid q-1$. Therefore, $3^6\mid u_m$, but then $3\mid m$, since $e_3=5$, a contradiction.
\end{case}
\begin{case}{$P=59$}
\newline
\newline
By \eqref{eqn12} and Lemma \ref{lem6}, we have
\begin{align*}
2\log\alpha&<\log\log\alpha+2.163+\sum_{\substack{d\mid n\\d>\alpha^2}}\frac{2\log d}{d}\\
&\leq\log\log\alpha+2.163+\sum_{\substack{d\mid 72n_1\\d>\alpha^2}}\frac{2\log d}{d}\\
&\leq\log\log\alpha+2.163-\log 2-\frac{2\log 3}{3}-\frac{\log 4}{2}-\frac{\log 6}{3}-\frac{\log 8}{4}-\frac{2\log 9}{9}+\sum_{d\mid 72n_1}\frac{2\log d}{d}\\
&<\log\log\alpha-1.561+\sum_{d\mid 72n_1}\frac{2\log d}{d}.
\end{align*}
Lemma \ref{lem8} gives
\begin{equation*}
2\log\alpha<\log\log\alpha-1.561+4\left(\log 2+\sum_{i=1}^3\frac{\log r_i}{r_i-1}\right)\prod_{i=1}^3\left(\frac{r_i}{r_i-1}\right).
\end{equation*}
If $r'\leq 167$ is a prime factor of $n_1$ not equal to $M$, then we must have $e_{r'}\geq 4$, for the same reason that $e_{r_1}\geq 4$. Checking with Maple, we see that the prime $r_j=3$ is the only prime with this property. Therefore,
\begin{align*}
2\log\alpha&<\log\log\alpha-1.561+4\left(\log 2+\frac{\log 3}{2}+\frac{\log 173}{172}+\frac{\log\left(3\cdot 10^{30}+1\right)}{3\cdot 10^{30}}\right)\cdot\frac{3\cdot 173\cdot\left(3\cdot 10^{30}+1\right)}{2\cdot 172\cdot 3\cdot 10^{30}}\\
&<\log\log\alpha+6.118,
\end{align*}
if $M=1$, and 
\begin{align*}
10\log\alpha&<\log\log\alpha-1.561+4\left(\log 2+\frac{\log 3}{2}+\frac{\log 5}{4}+\frac{\log\left(3\cdot 10^{30}+1\right)}{3\cdot 10^{30}}\right)\cdot\frac{3\cdot 5\cdot\left(3\cdot 10^{30}+1\right)}{2\cdot 4\cdot 3\cdot 10^{30}}\\
&<\log\log\alpha+10.78,
\end{align*}
if $M=5$. Both imply that $\alpha<42$, so that $P\leq 42$, a contradiction.
\end{case}
\end{proof}
We now finish the proof of Theorem \ref{thm}. We have $t=2$, $r_j\leq 167$, and the largest prime factor of $n_1$ being larger than $3\cdot 10^{30}$. It follows that $M=1$, so that $r_j=r_1$ and $r_2=r_t>3\cdot 10^{30}$. We divide into cases.
\setcounter{case}{0}
\begin{case}{$r_1\geq 11$}
\newline
\newline
We know that $e_{r_1}\geq 2$. By Maple, we can deduce that $P\geq 5$, so that $\alpha^2>26$. By \eqref{eqn12} and Lemma \ref{lem6}, we have
\begin{align*}
2\log\alpha&<\log\log\alpha+2.163+\sum_{\substack{d\mid n\\d>\alpha^2}}\frac{2\log d}{d}\\
&\leq\log\log\alpha+2.163+\sum_{\substack{d\mid 16n_1\\d>26}}\frac{2\log d}{d}\\
&\leq\log\log\alpha+2.163-\log 2-\frac{\log 4}{2}-\frac{\log 8}{4}-\frac{\log 16}{8}+\sum_{d\mid 16n_1}\frac{2\log d}{d}\\
&<\log\log\alpha-0.089+\sum_{d\mid 16n_1}\frac{2\log d}{d}.
\end{align*}
Lemma \ref{lem8} gives
\begin{align*}
2\log\alpha&<\log\log\alpha-0.089+4\left(\log 2+\sum_{i=1}^3\frac{\log r_i}{r_i-1}\right)\prod_{i=1}^2\left(\frac{r_i}{r_i-1}\right)\\
&<\log\log\alpha-0.089+4\left(\log 2+\frac{\log 11}{10}+\frac{\log\left(3\cdot 10^{30}+1\right)}{3\cdot 10^{30}}\right)\cdot\frac{11\cdot\left(3\cdot 10^{30}+1\right)}{10\cdot 3\cdot 10^{30}}\\
&<\log\log\alpha+4.016.
\end{align*}
This implies that $\log\alpha<12$, so that $P\leq 11$. Maple tells us the only possibility is $P=5$ and $r_1=11$. We also use Maple to see that $e_{11}=2$, $z(11)=12$, and that $(643457\cdot 23)\mid 73997555=v_{11}$, where $23$ and $643457$ are prime. Therefore, $11^2\mid(643456\cdot 22)\mid u_m$. Also, $11r_2\mid n_1$, and we have $v_{11r_2}\mid u_n$. Also, $v_{11r_2}$ has a primitive prime factor, say $q$. Then $q\equiv 1\pmod{11r_2}$, so $11\mid q-1$. Therefore, $11^3\mid u_m$. But then $11\mid m$, since $e_3=2$, so that $11\mid M$, a contradiction.
\end{case}
\begin{case}{$r_1=7$}
\newline
\newline
We know that $e_7\geq 2$. By Maple, we can deduce that $P\geq 12$, so that $\alpha^2>145$. By \eqref{eqn12} and Lemma \ref{lem6}, we have
\begin{align*}
2\log\alpha&<\log\log\alpha+2.163+\sum_{\substack{d\mid n\\d>\alpha^2}}\frac{2\log d}{d}\\
&\leq\log\log\alpha+2.163+\sum_{\substack{d\mid 896n_1\\d>145}}\frac{2\log d}{d}\\
&\leq\log\log\alpha+2.163-\log 2-\frac{\log 4}{2}-\frac{\log 8}{4}-\frac{\log 16}{8}-\frac{\log 32}{16}-\frac{\log 64}{32}-\frac{\log 128}{64}\\
&\quad-\frac{2\log 7}{7}-\frac{\log 14}{7}-\frac{\log 28}{14}-\frac{\log 56}{28}-\frac{\log 112}{56}-\frac{2\log 49}{49}-\frac{\log 98}{49}+\sum_{d\mid 896n_1}\frac{2\log d}{d}\\
&<\log\log\alpha-2.163+\sum_{d\mid 896n_1}\frac{2\log d}{d}.
\end{align*}
Lemma \ref{lem8} gives
\begin{align*}
2\log\alpha&<\log\log\alpha-2.163+4\left(\log 2+\sum_{i=1}^3\frac{\log r_i}{r_i-1}\right)\prod_{i=1}^2\left(\frac{r_i}{r_i-1}\right)\\
&<\log\log\alpha-2.163+2\left(\log 2+\frac{\log 7}{6}+\frac{\log\left(3\cdot 10^{30}+1\right)}{3\cdot 10^{30}}\right)\cdot\frac{2\cdot 7\cdot\left(3\cdot 10^{30}+1\right)}{6\cdot 3\cdot 10^{30}}\\
&<\log\log\alpha+2.586.
\end{align*}
This implies that $\alpha<5$, so that $P\leq 4$, a contradiction to $P\geq 12$.
\end{case}
\begin{case}{$r_1=5$}
\newline
\newline
We know that $e_5\geq 2$. By Maple, we can deduce that $P\geq 7$, so that $\alpha^2>50$. By \eqref{eqn12} and Lemma \ref{lem6}, we have
\begin{align*}
2\log\alpha&<\log\log\alpha+2.163+\sum_{\substack{d\mid n\\d>\alpha^2}}\frac{2\log d}{d}\\
&\leq\log\log\alpha+2.163+\sum_{\substack{d\mid 800n_1\\d>50}}\frac{2\log d}{d}\\
&\leq\log\log\alpha+2.163-\log 2-\frac{\log 4}{2}-\frac{\log 8}{4}-\frac{\log 16}{8}-\frac{\log 32}{16}-\frac{2\log 5}{5}\\
&\quad-\frac{\log 10}{5}-\frac{\log 20}{10}-\frac{\log 40}{20}-\frac{2\log 25}{25}-\frac{\log 50}{25}+\sum_{d\mid 800n_1}\frac{2\log d}{d}\\
&<\log\log\alpha-2.308+\sum_{d\mid 800n_1}\frac{2\log d}{d}.
\end{align*}
Lemma \ref{lem8} gives
\begin{align*}
2\log\alpha&<\log\log\alpha-2.308+4\left(\log 2+\sum_{i=1}^3\frac{\log r_i}{r_i-1}\right)\prod_{i=1}^2\left(\frac{r_i}{r_i-1}\right)\\
&<\log\log\alpha-2.308+2\left(\log 2+\frac{\log 5}{4}+\frac{\log\left(3\cdot 10^{30}+1\right)}{3\cdot 10^{30}}\right)\cdot\frac{2\cdot 5\cdot\left(3\cdot 10^{30}+1\right)}{4\cdot 3\cdot 10^{30}}\\
&<\log\log\alpha+3.17.
\end{align*}
This implies that $\alpha<7$, so that $P\leq 6$, a contradiction to $P\geq 7$.
\end{case}
\begin{case}{$r_1=3$}
\newline
\newline
We know that $e_3\geq 2$. By Maple, we can deduce that $P\geq 4$, so that $\alpha^4>17$. By \eqref{eqn12} and Lemma \ref{lem6}, we have
\begin{align*}
2\log\alpha&<\log\log\alpha+2.163+\sum_{\substack{d\mid n\\d>\alpha^2}}\frac{2\log d}{d}\\
&\leq\log\log\alpha+2.163+\sum_{\substack{d\mid 48n_1\\d>17}}\frac{2\log d}{d}\\
&\leq\log\log\alpha+2.163-\log 2-\frac{\log 4}{2}-\frac{\log 8}{4}-\frac{\log 16}{8}-\frac{2\log 3}{3}-\frac{\log 6}{3}\\
&\quad-\frac{\log 12}{6}-\frac{2\log 9}{9}+\sum_{d\mid 48n_1}\frac{2\log d}{d}\\
&<\log\log\alpha-2.321+\sum_{d\mid 800n_1}\frac{2\log d}{d}.
\end{align*}
Lemma \ref{lem8} gives
\begin{align*}
2\log\alpha&<\log\log\alpha-2.321+4\left(\log 2+\sum_{i=1}^3\frac{\log r_i}{r_i-1}\right)\prod_{i=1}^2\left(\frac{r_i}{r_i-1}\right)\\
&<\log\log\alpha-2.321+2\left(\log 2+\frac{\log 3}{2}+\frac{\log\left(3\cdot 10^{30}+1\right)}{3\cdot 10^{30}}\right)\cdot\frac{2\cdot 3\cdot\left(3\cdot 10^{30}+1\right)}{2\cdot 3\cdot 10^{30}}\\
&<\log\log\alpha+5.134.
\end{align*}
This implies that $\alpha<23.1$, so that $P\leq 22$. We have already dealt with the case $P=22$, so we may assume that $P\leq 21$. According to Maple, the possibilities of $P$ are $4,5,9,13,14$, and $18$. For the cases of $P=4,9,13,14,18$, we use Maple to see that $e_3=2$. In the case of $P=13$, we can see that there exists two primes that are $\equiv 1\pmod 3$ and that divide $v_3$ ($13$ and $43$). In the cases of $P=4$ and $P=14$, we can see that there exists a prime that is $\equiv 1\pmod 9$ and that divides $v_3$ ($19$ for $P=4$ and $199$ for $P=14$). In the cases of $P=9$ and $P=18$, we can see that $27\mid v_3$. So in all cases, except $P=5$, we have $9\mid\varphi\left(v_3\right)$. Therefore, we can repeat the argument given for $P=22$ to see that this leads to $3\mid M$, a contradiction. It remains to deal with $P=5$.
\newline
\newline
For $P=5$, we note that $e_3=3$. So, since $3\nmid m$, this means that $81\nmid u_m$. $v_{3r_2}$ must have a primitive prime factor, say $q$, such that $q\equiv 1\pmod{3r_2}$. So $3\mid q-1$. We have $6r_2\mid n$. Suppose that $6r_2<n$. Then there exists a prime factor, say $r'$, such that $r'\mid\frac{n}{6r_2}$. If $r'=2$, then we have $3^5\mid\varphi\left(u_{12}\right)\mid u_m$, leading to $3^4\mid u_m$, a contradiction. So $r'$ is odd. But then $v_{3r_2r'}$ has a primitive prime factor, say $q'$, and we can see that $q'\equiv 1\pmod 3$. Therefore, $81\mid\varphi\left(u_6\right)(q-1)\left(q'-1\right)\mid u_m$, again, a contradiction. Therefore, $n=6r_2$. By Lemmas \ref{lem2} and \ref{lem9}, we can see that $u_6=3640$ and $\frac{u_{6r_2}}{u_6}$ must be coprime. We have
\begin{equation*}
26.96<\alpha^2<\frac{u_n}{u_m}=\frac{u_6}{\varphi\left(u_6\right)}\cdot\frac{\displaystyle\frac{u_{6r_2}}{u_6}}{\varphi\left(\displaystyle\frac{u_{6r_2}}{u_6}\right)}<3.16\cdot\frac{\displaystyle\frac{u_{6r_2}}{u_6}}{\varphi\left(\displaystyle\frac{u_{6r_2}}{u_6}\right)}.
\end{equation*}
If $p$ is a prime factor of $\frac{u_{6r_2}}{u_6}$, then $p$ must be a primitive prime factor of $u_{r_2}$, $u_{2r_2}$, $u_{3r_2}$, or $u_{6r_2}$. In all cases, we have $p\equiv\pm 1\pmod{r_2}$, so that $r_2\leq p+1$. Therefore,
\begin{equation*}
\frac{\displaystyle\frac{u_{6r_2}}{u_6}}{\varphi\left(\displaystyle\frac{u_{6r_2}}{u_6}\right)}\leq\left(1+\frac{1}{r_2-2}\right)^y
\end{equation*}
where $y$ is the number of distinct prime factors of $\frac{u_{6r_2}}{u_6}$. We also have $2^y\mid u_m$. Then we have $2^{y-2}<m<n=6r_2$. Then
\begin{equation*}
y-2<\frac{\log 6+\log r_2}{\log 2},
\end{equation*}
or
\begin{equation*}
y<\frac{\log 24+\log r_2}{\log 2}.
\end{equation*}
So we have
\begin{equation*}
26.96<3.16\cdot\left(1+\frac{1}{r-2}\right)^{\frac{\log 24+\log r_2}{\log 2}}<3.16e^{\frac{\log 24+\log r_2}{(r-2)\log 2}}<3.16e^{\frac{\log 24+\log\left(3\cdot 10^{30}+1\right)}{\left(3\cdot 10^{30}-1\right)\log 2}}<3.17,
\end{equation*}
a contradiction. 
\end{case}
This finishes the proof of Theorem \ref{thm}.
\section{Future Work}
There are various future directions this research could go in. Here we managed to find all of the solutions to $\varphi\left(u_n\right)=u_m$ where $\left(u_n\right)_n$ is a Lucas sequence of the first kind with recurrence relation $u_n=Pu_{n-1}+u_{n-2}$. However, the general recurrence relation for a Lucas sequence is $u_n=Pu_{n-1}+Qu_{n-2}$, so it would be interesting to investigate the case of $Q\neq 1$. As well, we may replace $\left(u_n\right)_n$ with a Lucas sequence of a second kind, or any other binary recurrence sequence. We can also ask what happens when we replace the Euler totient function with another arithmetic function, such as the sum of the divisors function or the sum of the $k$th powers of the divisors function (Luca managed to do this with the regular Fibonacci sequence \cite{luca3}, but we can also ask this question for other binary recurrence sequences as well). 
\section*{Acknowledgements}
The author would like to thank the University of Calgary for the award of a postdoctoral fellowship, which made this research possible.

\end{document}